\documentclass [twoside,reqno,  draft, 11pt] {amsart}

\usepackage{tikz}

\usepackage{amsfonts}
\usepackage{amssymb}
\usepackage{a4}
\usepackage{color}

\newtheorem{thm}{Theorem}[section]
\newtheorem{cor}[thm]{Corollary}
\newtheorem{lem}[thm]{Lemma}
\newtheorem{prop}[thm]{Proposition}

\newtheorem{rem}[thm]{Remark}

\theoremstyle{definition}

\numberwithin{equation}{section}

\renewcommand{\Re}{\hbox{Re}\,}
\renewcommand{\Im}{\hbox{Im}\,}

\newcommand{\C}{\mathbb{C}}

\newcommand{\N}{\mathbb{N}}

\newcommand{\R}{\mathbb{R}}

\newcommand{\supp}{\operatorname{supp}}

\def\hat{\widehat}
\def\tilde{\widetilde}
\def \bfo {\begin {eqnarray*} }
\def \efo {\end {eqnarray*} }
\def \ba {\begin {eqnarray*} }
\def \ea {\end {eqnarray*} }
\def \beq {\begin {eqnarray}}
\def \eeq {\end {eqnarray}}
\def \supp {\hbox{supp }}

\def \p {\partial}

\def\hat{\widehat}
\def\tilde{\widetilde}
\def \bfo {\begin {eqnarray*} }
\def \efo {\end {eqnarray*} }
\def \ba {\begin {eqnarray*} }
\def \ea {\end {eqnarray*} }
\def \beq {\begin {eqnarray}}
\def \eeq {\end {eqnarray}}
\def \supp {\hbox{supp }}

\def \p {\partial}

\begin{document}

 \title[From Strichartz estimates to uniform $L^p$ resolvent estimates]{From semiclassical Strichartz estimates to uniform $L^p$ resolvent estimates on compact manifolds}

\author[Burq]{Nicolas Burq}

\address
        {N. Burq, D\'epartement de math\'ematiques\\ 
Universit\'e Paris-Sud\\  
 Bat. 425, 91405 Orsay Cedex \\
France}

\email{nicolas.burq@math.u-psud.fr}

\author[Dos Santos Ferreira]{David Dos Santos Ferreira}

\address
        {D. Dos Santos Ferreira, Institut \'Elie Cartan 
Universit\'e de Lorraine \\
B.P. 70239\\ 
F-54506 Vandoeuvre-l\`es-Nancy Cedex\\ 
France }

\email{ddsf@math.cnrs.fr}

\author[Krupchyk]{Katya Krupchyk}

\address
        {K. Krupchyk, Department of Mathematics\\
University of California, Irvine\\ 
CA 92697-3875, USA }

\email{katya.krupchyk@uci.edu}

\begin{abstract}

We prove uniform $L^p$ resolvent estimates for the stationary damped wave operator. The uniform $L^p$ resolvent estimates for  the Laplace operator on a compact smooth Riemannian manifold without boundary were first  established by Dos Santos Ferreira--Kenig--Salo \cite{DKS_resolvent} and advanced further by Bourgain--Shao--Sogge--Yao \cite{Bourgain_Shao_Sogge_Yao}. Here we provide an alternative proof relying on the techniques of semiclassical Strichartz estimates. This  approach allows us also to handle  non-self-adjoint perturbations of the Laplacian and embeds very naturally in the semiclassical spectral analysis framework. 

\end{abstract}

\maketitle

\section{Introduction and statement of result}

\label{sec_int}

Let $(M, g)$ be a compact smooth connected Riemannian manifold of dimension $n\ge 3$ without boundary, and let $-\Delta_g$ be the Laplace operator associated to the metric $g$.  The operator $-\Delta_g$ is self-adjoint on $L^2(M)$ with the domain $H^2(M)$, the standard Sobolev space on $M$, and it has a discrete spectrum $\text{Spec}(-\Delta_g)\subset[0,\infty)$. In \cite{DKS_resolvent}, see also \cite{Bourgain_Shao_Sogge_Yao} and  \cite{Shen_2001}, the following uniform resolvent estimates for  the Laplace operator were established. 

\begin{thm}
\label{thm_main}
Given $\delta>0$,
there exists a constant $C=C(\delta)>0$ such that  for all $u\in C^\infty(M)$ and all $\lambda\in \mathcal{R}_\delta$, we have
\begin{equation}
\label{eq_resolvent_est_laplacian}
\|u\|_{L^{\frac{2n}{n-2}}(M)}\le C\|(-\Delta_g-\lambda)u\|_{L^{\frac{2n}{n+2}}(M)},
\end{equation}
where
\begin{equation}
\label{eq_region_R_delta}
\mathcal{R}_\delta=\{\lambda\in \C: (\emph{\Im} \lambda)^2\ge 4\delta^2(\emph{\textrm{Re}}\,\lambda+\delta^2)\}.
\end{equation}
\end{thm}

Notice that $\mathcal{R}_\delta$ is the exterior of a parabolic region, containing the spectrum of $-\Delta_g$,   see Figure \ref{pic_laplacian}. Letting $\lambda=z^2$, $\Im z>0$, we observe that the region $\mathcal{R}_\delta$ is the image of the region $\Xi_\delta=\{z\in \C: \Im z\ge \delta\}$ under the map $z\mapsto z^2$, and the estimate  \eqref{eq_resolvent_est_laplacian} is equivalent to 
\begin{equation}
\label{eq_resolvent_est_laplacian_z} 
\|u\|_{L^{\frac{2n}{n-2}}(M)}\le C \|(-\Delta_g-z^2)u\|_{L^{\frac{2n}{n+2}}(M)}, \quad z\in\Xi_\delta,
\end{equation}
see Figure \ref{pic_laplacian}. 
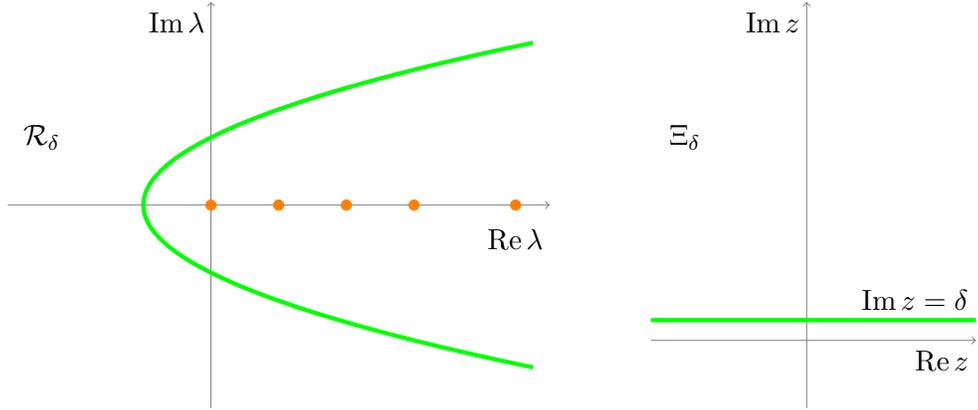
\begin{figure} [ht]
\centering
\begin{tikzpicture}[scale=.9]
\draw [help lines, ->] (-1,2) -- (7,2);
\draw [help lines,  ->] (2,-1) -- (2,5);
\draw[green, ultra thick, domain=2:4.4] plot ({(\x-2)*(\x-2)+1}, \x);
\draw[green, ultra thick, domain=2:4.4] plot ({(\x-2)*(\x-2)+1}, 4-\x);
\node at (6.5,1.5) {$\Re \lambda$};
\node at (1.5, 4.7) {$\Im  \lambda$};
\node at (-0.5, 3) {$\mathcal{R}_\delta$};
\fill[orange] (2,2) circle (0.5ex);
\fill[orange] (3,2) circle (0.5 ex);
\fill[orange] (4,2) circle (0.5 ex);
\fill[orange] (5,2) circle (0.5 ex);
\fill[orange] (6.5,2) circle (0.5 ex);
\draw [help lines, ->] (8.5,0) -- (13.3,0);
\draw [help lines,  ->] (10.8,-1) -- (10.8,5);
\draw [green,  ultra thick, -] (8.5,0.3) -- (13.3,0.3);
\node at (12.8,-0.3) {$\Re z$};
\node at (10.3, 4.7) {$\Im z$};
\node at (9, 3) {$\Xi_\delta$};
\node at (12.4,0.6) {$\Im z=\delta$};
\end{tikzpicture}
\caption{The spectral regions $\mathcal{R}_\delta$ and $\Xi_\delta$  the uniform resolvent bound \eqref{eq_resolvent_est_laplacian} and \eqref{eq_resolvent_est_laplacian_z}, respectively.}
\label{pic_laplacian}
\end{figure}

The work  \cite{Bourgain_Shao_Sogge_Yao} has addressed the very interesting question of the optimality of the spectral region 
$\mathcal{R}_\delta$, for the uniform estimate \eqref{eq_resolvent_est_laplacian} to hold, revealing that any sharpening in the spectral region is related to improvements in the remainder estimates in the Weyl law for the Laplacian. Thus, in \cite{Bourgain_Shao_Sogge_Yao} it was shown 
that the parabolic region $\mathcal{R}_\delta$ cannot be improved when $M$  is  the standard sphere, or more generally,  a Zoll manifold, 
due to  a cluster structure of the spectrum of $-\Delta_g$ on such manifolds, \cite{Weinstein_1977}. Exploiting the link to the Weyl law, 
the work \cite{Bourgain_Shao_Sogge_Yao} also obtained an improvement  over the spectral region in the case of manifolds of nonpositive
sectional curvature and in the case of flat tori.

In \cite{Uhlmann_Krupchyk_res} the uniform resolvent estimates \eqref{eq_resolvent_est_laplacian}  were extended to the case of higher order elliptic self-adjoint differential operators. 

The uniform $L^p$ resolvent estimates \eqref{eq_resolvent_est_laplacian}  have turned out to be  a very useful tool in the study of the structure of the spectrum of the periodic Schr\"odinger operator  with an unbounded potential, see \cite{Shen_2001} and \cite{Uhlmann_Krupchyk_abs}.  Such estimates are also crucial in solving inverse boundary value problems for Schr\"odinger operators  with unbounded potentials, see \cite{DKS_inverse} and \cite{Uhlmann_Krupchyk_inverse}. Uniform resolvent estimates in various $L^p$ spaces are also of great importance in control theory, see \cite{Bourgain_Burq_Zworski}.   

To state the main result of this paper, let us consider the following damped wave equation, 
\[
(\p_t^2-\Delta_g+2a(x)\p_t) v(x,t)=0, \quad (x,t)\in M\times\R,
\]  
arising in control theory, see \cite{Lebeau_1996}. Here $a\in C^\infty(M; \R)$ is the damping coefficient. Searching for solutions of the form $v(x,t)=e^{it\tau}u(x)$, $\tau\in \C$, we are led to the corresponding non-self-adjoint spectral  problem,
\begin{equation}
\label{eq_itro_damped_wave}
P(\tau)u:=(-\Delta_g+2i a(x)\tau -\tau^2)u(x)=0.
\end{equation}
We say that $\tau\in \C$ is an eigenvalue of $P(\tau)$, if there exists a corresponding non-trivial function $u \in L^2(M)$ 
such that $P(\tau)u = 0$.  

Using an integration by parts argument, one can easily see that the eigenvalues of $P(\tau)$ are confined to a band parallel to the real axis, see \cite{Sjostrand_2000}. More precisely, if $\tau$ is an eigenvalue, then we have 
\begin{equation}
\label{eq_intro_localization_eigenvalues}
\begin{aligned}
\inf a\le &\Im \tau\le \sup a,\quad \text{when}\quad \Re\tau\ne 0,\\
2\min(\inf a,0)\le &\Im \tau\le 2 \max (\sup a,0),\quad  \text{when}\quad \Re\tau= 0.
\end{aligned}
\end{equation} 
We shall denote the set of the eigenvalues of $P(\tau)$ by $\text{Spec}(P(\tau))$.  Using the analytic Fredholm theory, applied to the family of the operators $P(\tau):H^2(M)\to L^2(M)$, we see that the set $\text{Spec}(P(\tau))$ is discrete.  Furthermore, since \eqref{eq_itro_damped_wave} is invariant under the map $(\tau,u)\mapsto (-\overline{\tau},\overline{u})$, the eigenvalues are located symmetrically around the imaginary axis.

It turns out that the bounds \eqref{eq_intro_localization_eigenvalues}  on the imaginary parts of the eigenvalues of $P(\tau)$ can be sharpened in the regime when $|\Re \tau|$ is large, as has been established in \cite{Lebeau_1996},   \cite{Sjostrand_2000}, \cite{Koch_Tataru_1995}. Since the sharpened bounds are of major importance to us, we shall now recall the precise statement. First let us introduce some notation. Let $p(x,\xi)=|\xi|^2_g$ be the principal symbol of $-\Delta_g$, defined on $T^*M$, and let $H_p$ be the corresponding Hamilton vector field.  The Hamilton vector field $H_p$  generates the flow $\text{exp}(tH_p):T^*M \to T^*M$, $t\in \R$,  which is called the Hamilton flow of $p$, and which in this case can be identified with the geodesic flow, see \cite[Chapter 4]{Sternberng_book}. For $T>0$, let $\langle a\rangle_T$ denote the symmetric time $2T$ average of $a$ along the $H_p$--flow, 
\[
\langle a\rangle_T(x,\xi)=\frac{1}{2T}\int_{-T}^T a\circ \text{exp}(tH_p)(x,\xi)dt, \quad (x,\xi)\in T^*M.
\]
It was shown in \cite{Lebeau_1996} and \cite[Appendix A]{Sjostrand_1996} that 
\begin{align*}
A_+:&=\inf_{T>0}\sup_{p^{-1}(1)}\langle a\rangle_T=\lim_{T\to \infty}\sup_{p^{-1}(1)}\langle a\rangle_T,\\
A_-:&=\sup_{T>0}\inf_{p^{-1}(1)}\langle a\rangle_T=\lim_{T\to \infty}\inf_{p^{-1}(1)}\langle a\rangle_T.
\end{align*}
\begin{thm}[\cite{Lebeau_1996},  \cite{Sjostrand_2000},  \cite{Koch_Tataru_1995}] 
\label{thm_Lebeau_Sjostrand}
For every $\varepsilon>0$, there are at most finitely many eigenvalues of $P(\tau)$ outside the strip $\R+i(A_--\varepsilon, A_++\varepsilon)$.
\end{thm}

In general we have $A_+\le \sup a$, $A_- \ge \inf a$, and the latter inequality becomes strict, for instance when 
$a\ge 0$, $\inf a=0$, and the geometric control condition of Rauch and Taylor holds. Here we say that the geometric control condition holds if there exists a time $T_0>0$ such that any geodesic of length $\ge T_0$ meets the open set 
$\{x\in M: a(x)>0\}$, see \cite{Rauch_Taylor_1975}. 

Our main result is the following generalization of the uniform estimate \eqref{eq_resolvent_est_laplacian_z} to the stationary damped wave operator.
\begin{thm}
\label{thm_main_2}
For any $\delta>0$, there exists a constant $L>0$, such that for any neighborhood $V$ of the set $\emph{\text{Spec}}(P(\tau))\cap \{\tau\in \C: |\emph{\Re}\tau|\le L\}$,  there exists a constant $C=C(\delta, V)>0$ such that we have the uniform estimate, 
\begin{equation}
\label{eq_resolvent_est_damped_wave_equation}
\|u\|_{L^{\frac{2n}{n-2}}(M)}\le C\|P(\tau)u\|_{L^{\frac{2n}{n+2}}(M)},
\end{equation}
for all $u\in C^\infty(M)$ and all  $\tau\in\Pi_{\delta,V}$. Here 
\begin{equation}
\label{eq_resolvent_est_damped_wave_equation_region}
\begin{aligned}
\Pi_{\delta,V}=&\{\tau\in \C: \emph{\Im} \tau\ge A_++\delta, |\emph{\Re} \tau|\ge L\}\\
&\cup \{\tau\in \C: \emph{\Im} \tau\le A_--\delta, |\emph{\Re} \tau|\ge L\}\\
&\cup 
\{\tau\in \C: |\emph{\Re}\tau|\le L\}\setminus V.
\end{aligned}
\end{equation}
\end{thm}
 See Figure \ref{pic_damped_wave_equation}, where the spectral region $\Pi_{\delta,V}$ is described. 

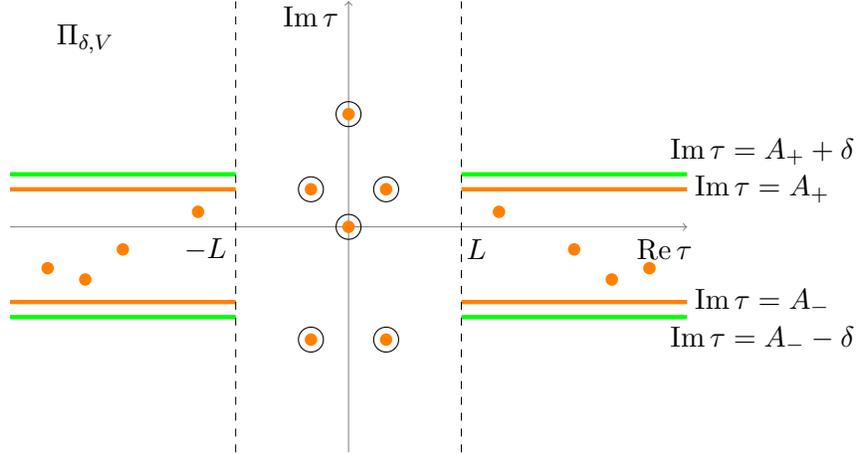
\begin{figure} [ht]
\centering
\begin{tikzpicture}
\draw [help lines, ->] (-3,0) -- (6,0);
\draw [help lines,  ->] (1.5,-3) -- (1.5,3);
\draw [orange,  ultra thick, -] (-3,0.5) -- (0,0.5);
\draw [orange,  ultra thick, -] (3,0.5) -- (6,0.5);
\node at (7,0.5) {$\Im \tau=A_+$};
\node at (7,1) {$\Im \tau=A_++\delta$};
\draw [green, ultra thick, -] (-3,0.7) -- (0,0.7);
\draw [green, ultra thick, -] (3,0.7) -- (6,0.7);
\draw [orange,  ultra thick, -] (-3,-1) -- (0,-1);
\draw [orange,  ultra thick, -] (3,-1) -- (6,-1);
\node at (7,-1) {$\Im \tau=A_-$};
\node at (7,-1.5) {$\Im \tau=A_--\delta$};
\draw [green, ultra thick, -] (-3,-1.2) -- (0,-1.2);
\draw [ green, ultra thick,-] (3,-1.2) -- (6,-1.2);
\draw [dashed] (0,-3) -- (0,3);
\node at (-0.4,-0.3) {$-L$};
\draw [dashed] (3,-3) -- (3,3);
\node at (3.2,-0.3) {$L$};
\node at (5.7,-0.3) {$\Re \tau$};
\node at (1, 2.8)  {$\Im \tau$};
\node at (-2, 2.5) {$\Pi_{\delta,V}$};
\draw (1.5,1.5) circle (1ex);
\fill[orange] (1.5,1.5) circle (0.5ex);
\draw (1.5,0) circle (1ex);
\fill[orange] (1.5,0) circle (0.5ex);
\draw (2,0.5) circle (1ex);
\fill[orange] (2,0.5) circle (0.5ex);
\draw (1,0.5) circle (1ex);
\fill[orange] (1,0.5) circle (0.5ex);
\draw (2,-1.5) circle (1ex);
\fill[orange] (2,-1.5) circle (0.5ex);
\draw (1,-1.5) circle (1ex);
\fill[orange] (1,-1.5) circle (0.5ex);
\fill[orange] (3.5,0.2) circle (0.5ex);
\fill[orange] (-0.5,0.2) circle (0.5ex);
\fill[orange] (4.5,-0.3) circle (0.5ex);
\fill[orange] (-1.5,-0.3) circle (0.5ex);
\fill[orange] (5,-0.7) circle (0.5ex);
\fill[orange] (-2,-0.7) circle (0.5ex);
\fill[orange] (5.5,-0.55) circle (0.5ex);
\fill[orange] (-2.5,-0.55) circle (0.5ex);
\end{tikzpicture}
\caption{Spectral region $\Pi_{\delta,V}$ in the uniform resolvent bound for the stationary damped wave equation \eqref{eq_resolvent_est_damped_wave_equation}.}
\label{pic_damped_wave_equation}
\end{figure}

\begin{rem}
\label{rem_main_2} 
In fact we shall prove the following more general result: 
for any $\delta>0$, there exists a constant $L>0$, such that for any neighborhood $V$ of the set $\emph{\text{Spec}}(P(\tau))\cap \{\tau\in \C: |\emph{\Re}\tau|\le L\}$,  there exists a constant $C=C(\delta, V)>0$ such that we have  
\begin{equation}
\label{eq_rem_main_2}
\|u\|_{L^{p'}(M)}\le C|\tau|^{2n(\frac{1}{p}-\frac{1}{2})-2} \|P(\tau)u\|_{L^{p}(M)},
\end{equation}
for all $u\in C^\infty(M)$, all  $\tau\in\Pi_{\delta,V}$,  and all $p\in [\frac{2n}{n+2}, \frac{2(n+1)}{n+3}]$. Here $\frac{1}{p}+\frac{1}{p'}=1$.
\end{rem}

\begin{rem} 
\label{rem_main_2_1} 
It is interesting and important to point out that using $L^2$ bounds for solutions of second order hyperbolic equations established in \cite{Koch_Tataru_1995}, $L^p$ resolvent estimates of Theorem \ref{thm_main_2}  and Remark \ref{rem_main_2} can be extended to 
the stationary wave operator perturbed by general first order terms, containing spatial derivatives. However, for simplicity of the exposition here we prefer to keep our discussion focused on the significant case of the damped wave equation.  
\end{rem} 

Notice that, in particular  in the region where $|\Re\tau|$ is large, Theorem \ref{thm_main_2} establishes uniform resolvent estimates for $P(\tau)$ outside of the dynamically defined strip  
$\R+i (A_--\delta, A_++\delta)$. After a semiclassical reduction, the first step in the proof of Theorem \ref{thm_main_2}  in a crucial region in the spectral $\tau$--plane given by 
$\Im\tau=A_++\delta$, $|\Re \tau|$ large, is therefore an averaging procedure realized by means of a pseudodifferential conjugation, which we carry out following \cite{Sjostrand_2000}. Once the averaging procedure has been performed, we obtain a semiclassical non-self-adjoint operator which is an $\mathcal{O}(h)$ perturbation of the operator $-h^2\Delta_g-1$, with the range of the symbol of the perturbation confined to the dynamical strip. To conclude, we develop an alternative proof  of Theorem \ref{thm_main} relying on the semiclassical Strichartz estimates~\cite{ST, BGT, Burq_Gerard_Tzvetkov_2004, Koch_Tataru_Zworski} which handles very easily such perturbations and embeds naturally in the semiclassical spectral analysis framework. Let us also mention the work \cite{Koch_Tataru_2005}, which gives a parametrix construction for principally normal operators, that could also be used to establish the Strichartz estimates, leading to our results. 

\begin{rem}
\label{rem_main}
In fact exploiting the semiclassical approach, we shall give a proof of the following more general version of Theorem \ref{thm_main}: given $\delta>0$,
there exists a constant $C=C(\delta)>0$ such that  for all $u\in C^\infty(M)$, all $z\in\Xi_\delta$, and all $p\in [\frac{2n}{n+2}, \frac{2(n+1)}{n+3}]$, we have 
\begin{equation}
\label{eq_rem_main}
\|u\|_{L^{p'}(M)}\le C |z|^{2n(\frac{1}{p}-\frac{1}{2})-2}\|(-\Delta_g-z^2)u\|_{L^{p}(M)},
\end{equation}
where $\frac{1}{p}+\frac{1}{p'}=1$. The estimate \eqref{eq_rem_main} agrees with the corresponding resolvent bounds for the Euclidean Laplacian, valid in all of $\C$, see \cite{Kenig_Ruiz_Sogge}.
\end{rem}

The proof of Theorem \ref{thm_main} given in \cite{DKS_resolvent} is based on the Hadamard parametrix for the resolvent $(-\Delta_g-\lambda)^{-1}$, and some estimates for oscillatory integral operators obtained by the Carleson-Sj\"olin theory.  The proof  developed in 
\cite{Bourgain_Shao_Sogge_Yao} is based on the Hadamard parametrix for the operator $\cos (t\sqrt{-\Delta_g})$,  the spectral cluster estimates, obtained in \cite{Sogge_1988},  and the method of stationary phase. Here we shall  deduce Theorem \ref{thm_main} from  semiclassical Strichartz estimates.  In doing so, we observe that the most crucial region in the complex spectral $z$--plane where the uniform estimates should be obtained is given by $\Im z=\delta>0$, $|\Re z|$ is large, which corresponds to the boundary of the parabolic region in the $\lambda$--plane. Here we perform a semiclassical rescaling, which allows us to microlocalize the problem to the energy surface $p(x,\xi) = 1$, where $p(x,\xi)$ is the principal symbol of $-\Delta_g$. To handle the problem near the energy surface, we follow the idea of \cite{BGT} (see also \cite{Koch_Tataru_Zworski}), and perform a factorization of $p(x,\xi)-1$, to reduce the problem to an evolution equation in $n-1$ variables, for which semiclassical Strichartz estimates in their localized form can be applied. The uniform resolvent estimates in the complementary regions in the spectral $z$--plane follow directly by means of some elliptic a priori estimates combined with the Phragm\'en--Lindel\"of principle. 

It has to be noticed that, at the end of the day, all known proofs of Theorem~\ref{thm_main} (including ours) rely on the choice of a particular space direction which plays the role of ``time" and eventually on the use of the (classical) Hardy-Littlewood-Sobolev inequality (see also~\cite{MSS,SS} where similar arguments are used). However, our approach to prove  Theorem \ref{thm_main} appears to have several advantages. One of them is that it is directly applicable to more general elliptic operators, including some non-self-adjoint ones, and can be combined  very naturally with methods and ideas coming from semiclassical spectral theory. 

The paper is organized as follows. In Section \ref{sec_Prelim} we collect some preliminaries on the semiclassical analysis, evolution equations, and semiclassical Strichartz estimates. Section \ref{sec_Laplace} is devoted to the proof of Theorem \ref{thm_main}, while the estimate \eqref{eq_rem_main} in Remark \ref{rem_main} is established in Section \ref{sec_Laplace_Remark}. In Section \ref{sec_Damped} the argument is extended to the damped wave operator and Theorem \ref{thm_main_2} is proved. Finally Section \ref{sec_damped_Remark} discusses the proof of  
the estimate \eqref{eq_rem_main_2} in  Remark \ref{rem_main_2}.

\section{Preliminaries on semiclassical analysis and Strichartz estimates}

\label{sec_Prelim}

In this section we shall collect some well-known results of semiclassical analysis and semiclassical Strichartz estimates  following \cite{Burq_Gerard_Tzvetkov_2004, BGT, Koch_Tataru_Zworski, Dimassi_Sjostrand_book, Zworski_book}.

Let  $m\in\R$ and let
\[
S^m(\R^{2k})=\{a\in C^\infty (\R^{2k}): |\p_x^\alpha\p_\xi^\beta a(x,\xi)|
\le C_{\alpha, \beta} \langle \xi\rangle^{m-|\beta|}, \alpha,\beta\in \N^{k}\}
\]
be the Kohn--Nirenberg symbol class. When $a\in S^m(\R^{2k})$, one defines 
the semiclassical Weyl quantization of $a$ as follows, 
\[
a^w(x,hD_x)u(x)=\text{Op}_h^w(a)u(x)=\frac{1}{(2\pi h)^k}\int\!\!\!\int e^{\frac{i}{h}\langle x-y,\xi\rangle} a\bigg(\frac{x+y}{2},\xi\bigg) u(y) dy d\xi, 
\]
$u\in \mathcal{S}(\R^k)$, the Schwartz space of functions on $\R^k$.

Let $a\in C^\infty(\R_t,S^0(\R^{2k}))$ be real-valued.  
Then by \cite[Theorem 4.23]{Zworski_book}, the self-adjoint operator $a^w(t,x,hD_x)$ is bounded on $L^2(\R^k)$, uniformly in $h\in (0,1]$, locally uniformly in $t\in \R$. 

Letting $r\in \R$, consider the operator equation,
\begin{equation}
\label{eq_4_1}
\begin{aligned}
(hD_{t}-a^w(t, x,hD_{x})) F(t,r)&=0,\quad t\in \R,\\
F(r,r)&=I.
\end{aligned}
\end{equation}
Then by \cite[Theorem 10.1]{Zworski_book},  the equation \eqref{eq_4_1} is uniquely solved by a family $\{F(t,r)\}_{t\in\R}$ of  unitary operators on $L^2(\R^k)$. Furthermore, we have
\begin{equation}
\label{eq_4_2}
F(t,r)F(r,s)=F(t,s)
\end{equation}
for all $t,r,s\in \R$. This follows from the fact that for any $r,s\in \R$,  $\tilde F(t,s,r)=F(t,r)F(r,s)-F(t,s)$ solves the homogeneous equation 
\[
(hD_{t}-a^w(t, x,hD_{x})) \tilde F(t,s,r)=0,
\]
with the initial conditions $\tilde F(r,s,r)=0$, and hence, $\tilde F(t,s,r)=0$ for any $t\in \R$. The fact that the operator $F(t,r)$ is unitary and \eqref{eq_4_2} yield  that 
\[
F(t,r)^*=F(t,r)^{-1}=F(r,t).
\]
This together with \eqref{eq_4_2} implies that 
\begin{equation}
\label{eq_4_3}
F(t,r)F(s,r)^*=F(t,s). 
\end{equation}

Let $l\in \R$, and  consider now the following initial value problem,
\begin{align*}
(hD_{t}-a^w(t, x,hD_{x})) u(t,x)&=f(t,x),\quad (t,x)\in \R\times\R^k,\\
u(l,x)&=u_0(x).
\end{align*}
Then Duhamel's formula together with \eqref{eq_4_3}  give
\begin{equation}
\label{eq_4_4}
\begin{aligned}
u(t,x)&=F(t,l)u_0(x)+\frac{i}{h} \int_{l}^t F(t,s) f(s,x) ds\\
&=F(t,l)u_0(x)+\frac{i}{h} \int_{l}^t F(t,r)F(s,r)^* f(s,x) ds,
\end{aligned}
\end{equation}
with any $r\in\R$.

Let us now recall the semiclassical Strichartz estimates,  following \cite{BGT,  
Burq_Gerard_Tzvetkov_2004, Koch_Tataru_Zworski} and  \cite[Theorem 10.6]{Zworski_book}.  The exponent pair $(p,q)$ is called  sharp $\sigma$-admissible, $\sigma>0$, if 
\[
\frac{2}{p}+\frac{2\sigma}{q}=\sigma,\quad 2\le p\le \infty, \quad 1\le q\le \infty, \quad (p,q)\ne (2,\infty),
\]
If $G(t,x)$ is a function on $\R\times\R^k$, we use the following standard notation for its mixed norm,  
\[
\| G\|_{L^p_tL^q_x}=\bigg(\int_{\R}\| G (t,\cdot)\|_{L^q(\R^k)}^p dt\bigg)^{1/p}.
\]
If $\tilde q\in [1,\infty]$, we denote by $\tilde q'$ the H\"older conjugate exponent of $\tilde q$. Notice that below we shall {\em not} need the difficult end-point case ($p=2$) but only the much easier case $2<p \leq + \infty$. Notice also that these estimates have a long history in harmonic analysis which can be traced back to the works by Stein and Tomas~\cite{To},  Strichartz~\cite{Stri}, Greenleaf~\cite{Gr} (see also Stein restriction theorem~\cite{St86} and the abstract setting developed by Ginibre and Velo in~\cite{GiVe}).

We shall need the semiclassical Strichartz estimates for the microlocalized solution operator $F(t,r)$ to the time-dependent Schr\"odinger equation \eqref{eq_4_1}, i.e. for
\[
U(t,r):=\psi(t-r)\chi^w(x,hD_x)F(t,r),
\]
where $\chi\in C^\infty_0(\R^{2k})$ and $\psi\in L^\infty(\R)$ with compact support sufficiently close to $0$. Using the semiclassical parametrix construction for $U(t,r)$, the following result was established in \cite[Section 2.2]{BGT}, see also \cite{Burq_Gerard_Tzvetkov_2004, Koch_Tataru_Zworski}, \cite[Theorem 10.8]{Zworski_book}.

\begin{thm}
\label{thm_Zworski}
Let $I\subset\R$ be a compact interval.  Assume that 
\begin{equation}
\label{eq_3_3}
\begin{aligned}
\p_\xi^2 a(t,x,\xi) \text{ is nonsingular for all } (x,\xi)\in\supp(\chi)\\
 \text{ and for all } t \text{ in a neighborhood of }I.  
 \end{aligned}
\end{equation}
Then we have 
\[
\| U(t,r)U^*(s,r)\|_{L^1(\R^k)\to L^\infty(\R^k)}\le C h^{-k/2} |t-s|^{-k/2},
\]
for all $s,t\in \R$ and $r\in I$ with a constant independent of $r$. 
\end{thm}

As a consequence of the standard $TT^*$ argument, the Hardy-Littlewood-Sobolev inequality and Theorem \ref{thm_Zworski}, we state the following semiclassical Strichartz estimate, see \cite[Corollary 2.2]{BGT} and \cite{Koch_Tataru_Zworski}, \cite[Theorem 10.7]{Zworski_book}.

\begin{cor}
Let $I\subset\R$ be a compact interval and assume that the condition \eqref{eq_3_3} holds. Then we have 
\begin{equation}
\label{eq_3_4}
\sup_{r\in I} \bigg( \int_{\R}\|U(t,r)f\|^p_{L^q(\R^k)} dt\bigg)^{1/p}\le Ch^{-1/p}\|f\|_{L^2(\R^k)},
\end{equation}
\begin{equation}
\label{eq_3_5}
\sup_{r\in I}\bigg\| \int_{-\infty}^t U(t,r) U^*(s,r) G(s,\cdot)ds\bigg\|_{L^p_tL^q_x}\le C h^{\big(\frac{1}{\tilde p'}-\frac{1}{p}-1\big)}\|G\|_{L^{\tilde p'}_tL^{\tilde q'}_x},
\end{equation}
for  all sharp $k/2$-admissible pairs $(p,q)$ and $(\tilde p,\tilde q)$. 
\end{cor}

We shall need the following result, which is a consequence of the proof of \cite[Lemma 2.2]{Koch_Tataru_Zworski}. 
\begin{lem}
\label{lem_Zworski_lem_2_2}
Let $x=(x',x'')\in \R^k$, $x'\in \R^{k_1}$, $x''\in\R^{k_2}$, so that $k_1+k_2=k$, and let $a\in \mathcal{S}(T^*\R^{k_1})$. Then for $1\le q\le p\le \infty$ and $1\le r\le \infty$, we have
\[
\|a^w(x',hD_{x'})u(x',x'')\|_{L^p_{x'}L^r_{x''}}\le C h^{k_{1}(\frac{1}{p}-\frac{1}{q})})\| u\|_{L^q_{x'}L^r_{x''}}.
\]
\end{lem}

Recall from \cite{Koch_Tataru_Zworski} that $u=u(h)\in L^2(\R^k)$ is said to be microlocalized in a compact subset of $T^*\R^k$ if there exists $\chi\in C^\infty_0(T^*\R^k)$ such that 
\begin{equation}
\label{eq_3_6_0}
 u=\chi^w(x,hD_x)u+\mathcal{O}_{\mathcal{S}}(h^\infty)\|u\|_{L^2(\R^k)}.
\end{equation}
Here  and in what follows the notation  $\mathcal{O}_{\mathcal{S}}(h^\infty)$ stands for a function $g(x;h)\in \mathcal{S}(\R^k)$ such that 
\[
\sup_{x\in \R^k}|x^\beta \p_{x}^\alpha g(x;h)|=\mathcal{O}(h^\infty), \quad \text{for all }\alpha, \beta. 
\]

We shall also need the following result from 
\cite[Lemma 3.1]{Nonnenmacher_Sjostrand_Zworski_2014}. 
\begin{lem}
\label{lem_NSZ}
Let $u=u(h)\in L^2(\R^k)$ be  microlocalized in a compact subset of $T^*\R^k$. Then 
\begin{equation}
\label{eq_lem_NSZ_1}
\|u\|_{L^\infty_{x_1}L^2_{x'}}=\mathcal{O}(h^{-\frac{1}{2}})\|u\|_{L^2(\R^{k})},
\end{equation}
and there exists $\varphi\in C^\infty_0(T^*\R^{k-1})$ such that 
\begin{equation}
\label{eq_lem_NSZ_2}
 u(x_1,x')=\varphi^w(x',hD_{x'})u(x_1,\cdot)+\mathcal{O}_{\mathcal{S}(\R^k)}(h^\infty)\|u\|_{L^2(\R^k)}.
\end{equation}
\end{lem}

Let us next  recall the classical Sobolev embedding. When $1\le p\le \infty$ and $s\in\R$, we set 
\[
W^{s,p}(\R^n)=\{u\in \mathcal{S}'(\R^n): (1-\Delta)^{\frac{s}{2}}u\in L^p(\R^n)\}. 
\]
Let  $1<p\le q<\infty$ and $\frac{1}{p}-\frac{1}{q}=\frac{s}{n}$. Then we know that 
\begin{equation}
\label{eq_classical_Sobolev_emb}
W^{s,p}(\R^n)\subset L^q(\R^n), \quad L^p(\R^n)\subset W^{-s,q}(\R^n),
\end{equation}
and the inclusions are continuous. We shall need the following semiclassical version of the embeddings \eqref{eq_classical_Sobolev_emb}. 

\begin{lem}
\label{lem_semiclass_emd}
Let $1<p\le q<\infty$ and $\frac{1}{p}-\frac{1}{q}=\frac{s}{n}$. Then 
\begin{equation}
\label{eq_semiclass_emb}
\|u\|_{L^q(\R^n)}\le C h^{-s}\|u\|_{W^{s,p}_{\emph{scl}}(\R^n)}, \quad u\in C^\infty_0(\R^n), 
\end{equation}
where
\[
\|u\|_{W^{s,p}_{\emph{scl}}(\R^n)}=\|  (1-h^2\Delta)^{\frac{s}{2}}u\|_{L^p(\R^n)}.
\]
\end{lem}
\begin{proof}
It suffices to consider the case $s>0$.  
We have to show that 
\[
\|  (1-h^2\Delta)^{-\frac{s}{2}}v\|_{L^q(\R^n)}\le C h^{-s} \|v\|_{L^p(\R^n)}, \quad v\in \mathcal{S}(\R^n). 
\]
Let 
\[
K(x,y)=\frac{1}{(2\pi h)^n}\int e^{i(x-y)\cdot\xi/h}(1+\xi^2)^{-\frac{s}{2}}d\xi
\]
be the Schwartz kernel of $(1-h^2\Delta)^{-\frac{s}{2}}$. Noticing that $0<s<n$ and applying  \cite[Lemma 0.3.8]{Sogge_book_Fourier}, we get 
\[
|K(x,y)|\le Ch^{-s} |x-y|^{-n+s}, \quad x\ne y.
\]
The result follows as explained in the proof of \cite[Theorem 0.3.7]{Sogge_book_Fourier}. 
 \end{proof}

By duality from \eqref{eq_semiclass_emb}, we get  
\begin{equation}
\label{eq_semiclass_emb_dual}
\|u\|_{W^{-s,q}_{\text{scl}}(\R^n)}\le Ch^{-s}\|u\|_{L^p(\R^n)}, \quad u\in C^\infty_0(\R^n).
\end{equation}
Here 
\[
\|u\|_{W^{-s,q}_{\text{scl}}(\R^n)}=\sup_{0\ne v\in W^{s,p}(\R^n)}\frac{|\langle u,v\rangle_{W^{-s,q}_{\text{scl}}(\R^n), W^{s,p}_{\text{scl}}(\R^n)} |}{\|v\|_{W^{s,p}_{\text{scl}}(\R^n)}}.
\]

When proving Theorem \ref{thm_main}, we shall need the calculus of semiclassical pseudodifferential operators on a compact smooth Riemannian manifold  $M$. Let us proceed by recalling some definitions and facts about them, following \cite[Chapter 14]{Zworski_book}. First recall the standard class of symbols on $T^*M$,
\[
S^m(T^*M)=\{a(x,\xi;h)\in C^\infty(T^*M\times(0,1]):|\p_x^\alpha\p_\xi^\beta a(x,\xi;h)|\le C_{\alpha\beta}\langle \xi\rangle^{m-|\beta|}\},
\]
$m\in \R$.  Let us fix a choice of the quantization map 
\[
\text{Op}_h^w: S^m(T^*M)\to \Psi^m(M), 
\]
given by the Weyl quantization in local coordinate charts, identified with convex domains in $\R^n$, with the associated symbol map,
\[
\sigma: \Psi^m(M)\to S^m(T^*M)/hS^{m-1}(T^*M).
\]

We have the following properties, enjoyed by the semiclassical pseudodifferential operators on $M$. 
\begin{prop} 
\label{prop_composition_of_operators}
Assume that $a\in S^{m_1}(T^*M)$ and $b\in S^{m_2}(T^*M)$. Then 
\begin{itemize}
\item[(i)] 
\[
\emph{\text{Op}}_h^w (a)\emph{\text{Op}}_h^w(b)-\emph{\text{Op}}_h^w(ab)\in h\emph{\text{Op}}_h^w(S^{m_1+m_2-1}),\]
\item[(ii)] 
\[
[\emph{\text{Op}}_h^w (a),\emph{\text{Op}}_h^w(b)]-\frac{h}{i}\emph{\text{Op}}_h^w(H_a(b))\in h^2\emph{\text{Op}}_h^w(S^{m_1+m_2-2}),
\]
where 
\[
H_a=\nabla_\xi a\cdot \nabla_x-\nabla_x a\cdot\nabla_\xi
\]
is the Hamilton vector field of $a$,
\item[(iii)] 
\[
( \emph{\text{Op}}_h^w (a))^*-\emph{\text{Op}}_h^w (\overline{a})\in h \emph{\text{Op}}_h^w (S^{m_1-1}),
\]
where $( \emph{\text{Op}}_h^w (a))^*$ is the formal $L^2$-adjoint of the operator $\emph{\text{Op}}_h^w (a)$.
\end{itemize}
\end{prop}

Let us also recall the semiclassical microlocalized version of G{\aa}rding's inequality, see \cite{Zworski_book}.
\begin{thm}
\label{thm_Garding_ineq}
Let $\Omega\subset T^*M$ be open bounded and let $a\in S^0(T^*M)$ be such that
\[
a\ge \gamma_0>0\quad \text{on}\quad \Omega.
\]
Let $\chi\in C^\infty_0(\Omega)$. Then for all $h>0$ small enough, and $u\in L^2(M)$, we have
\[
( \emph{\text{Op}}_h^w (a) \emph{\text{Op}}_h^w (\chi)u,  \emph{\text{Op}}_h^w (\chi)u)_{L^2(M)}\ge \frac{\gamma_0}{2}\| \emph{\text{Op}}_h^w (\chi)u\|^2_{L^2(M)}-\mathcal{O}(h^\infty)\|u\|^2_{L^2(M)}.
\]
\end{thm}

\begin{proof}
Let $\psi\in C^\infty(T^*M;[0,1])$ be such that $\psi=0$ near $\supp (\chi)$ and $\psi=1$ near $T^*M\setminus\Omega$.  Setting $\tilde a=a +C \psi\in S^0(T^*M)$ with the constant $C=\gamma_0+\sup_{T^*M}|a(x,\xi)|$, we have that $\tilde a=a$ near $\supp(\chi)$ and $\tilde a\ge \gamma_0>0$ on $T^*M$.  Applying G{\aa}rding's inequality \cite[Theorem 4.30]{Zworski_book} to $\tilde a$, we get
\[
(\text{Op}_h^w (\tilde a) \text{Op}_h^w (\chi)u, \text{Op}_h^w (\chi)u)_{L^2(M)}\ge \frac{\gamma_0}{2}\| \text{Op}_h^w (\chi)u\|^2_{L^2(M)},
\] 
for all $h>0$ small enough. 
This together with the fact that 
\[
(\text{Op}_h^w (\tilde a)-\text{Op}_h^w (a))\text{Op}_h^w (\chi) =\mathcal{O}(h^\infty):L^2(M)\to L^2(M)
\]
shows the claim. 
\end{proof}

\section{Laplace operator. Proof of Theorem \ref{thm_main}}

\label{sec_Laplace}

As a warmup, we shall give in this section a proof of Theorem~\ref{thm_main}, using a semiclassical point of view.
When  establishing  the resolvent estimate \eqref{eq_resolvent_est_laplacian_z},   for $u\in C^\infty(M)$, we write
\begin{equation}
\label{eq_2_1}
(-\Delta_g-z^2)u=f.
\end{equation}
The proof of the estimate \eqref{eq_resolvent_est_laplacian_z} will consist of several different cases, depending on the location of the spectral parameter $z$ in the region $\Xi_\delta=\{z\in \C: \Im z\ge \delta\}$.  

\subsection{Easy spectral regions}

 As the following proposition shows, in some cases the uniform resolvent estimate  \eqref{eq_resolvent_est_laplacian_z} is a direct consequence of \textit{a priori} estimates for the equation \eqref{eq_2_1}. 

\begin{prop}
There exists $C>0$ such that for any $z\in \mathbb{C}$,  $\emph{\Im} (z^2) \neq 0$,  we have  
\begin{equation}
\label{eq.ellipt1}
 \|u\|_{L^{\frac{ 2n} { n-2}}(M)}\leq C\bigg(\frac{|z|^2 +1} { |\emph{\Im} (z^2)|}+1\bigg)   \|(-\Delta_g-z^2) u\|_{L^{\frac{ 2n} { n+2}}(M)},
\end{equation}
and for any $z\in \mathbb{C}$,  $\emph{\Re} (z^2) < 0$,  we have
\begin{equation}\label{eq.ellipt2}
 \|u\|_{L^{\frac{ 2n} {n-2}}(M)}
\leq C\frac{1} { \min (1, - \emph{\Re} (z^2))}   \|(-\Delta_g-z^2) u\|_{L^{\frac{ 2n} { n+2}}(M)}.
\end{equation}

\end{prop}

\begin{proof}
Using the formulation of \eqref{eq_2_1} in terms of quadratic forms, we get   
\begin{equation}
\label{eq_elliptic_3}
\|\nabla_g u\|_{L^2(M)}^2-z^2\|u\|_{L^2(M)}^2=(f,u)_{L^2(M)},
\end{equation}
where $\nabla_g$ is the gradient operator with respect to the metric $g$. 
We deduce that
\[
\begin{cases}
 & |\Im (z^2)|  \| u\|_{L^2(M)} ^2 \leq | ( f, u)_{L^2(M)}| \leq \| u\|_{H^1(M)} \| f\|_{H^{-1}(M)},\\
 &\| \nabla_g u\|_{L^2(M)}^2 \leq | z|^2\| u\|_{L^2(M)} ^2 +    \| u\|_{H^1(M)} \| f\|_{H^{-1}(M)},
 \end{cases} 
 \]
 and therefore, 
 \[
 \| u\|_{H^1}^2 = \| \nabla_g u\|_{L^2}^2 + \| u\|_{L^2} ^2 \leq  \bigg(\frac{|z|^2 +1} { |\Im (z^2)|}+1\bigg)  \| u\|_{H^1} \| f\|_{H^{-1}}.
 \]
This bound together with the classical Sobolev embedding  $H^1(M)\hookrightarrow  L^{\frac{ 2n} {n-2}}(M)$, and its dual $L^{\frac {2n} {n+2} }(M) \hookrightarrow H^{-1}(M)$ imply the estimate \eqref{eq.ellipt1}. 

To get \eqref{eq.ellipt2},  we conclude from \eqref{eq_elliptic_3} that 
\[
\min(1,-\Re(z^2))\|u\|_{H^1(M)}^2\leq\|\nabla_g u\|_{L^2(M)}^2-\Re (z^2)\|u\|_{L^2(M)}^2=\Re (f,u)_{L^2(M)},
\]
 and proceed similarly to the derivation of \eqref{eq.ellipt1}. 
 \end{proof}

Now when $z\in \Xi_\delta$ and $\Im z\ge 2|\Re z|$,  writing $z^2=(\Re z)^2-(\Im z)^2+2i \Re z\Im z$, we see that  the uniform resolvent estimate \eqref{eq_resolvent_est_laplacian_z} in this region is a consequence of the a priori estimate \eqref{eq.ellipt2}. Next, in the region where $\Im z=\delta$ and $\frac{\delta}{2}\le |\Re z|\le C$, with $C>0$ being a constant, the uniform resolvent estimate \eqref{eq_resolvent_est_laplacian_z}  follows from the a priori estimate \eqref{eq.ellipt1}. 

When establishing the uniform resolvent estimate \eqref{eq_resolvent_est_laplacian_z},  the most crucial region is therefore given by $\Im z=\delta$, $|\text{Re}\,z|$ is large. Assuming that we have proved \eqref{eq_resolvent_est_laplacian_z} in this region, let us conclude the proof of Theorem \ref{thm_main}.  To that end, it remains to establish the estimate \eqref{eq_resolvent_est_laplacian_z}  in the sectors $\delta< \Im z< 2 \Re z$ and $\delta< \Im z<  -2 \Re z$. Without loss of generality we shall consider the sector  $\Pi=\{z\in \C:\delta< \Im z< 2 \Re z\}$.  Let $u,v\in C^\infty(M)$ be fixed and let 
\[
F(z)=( (-\Delta_g-z^2)^{-1} u, v)_{L^2(M)}, \quad z\in \Pi.
\]
The function $F$ is analytic in $\Pi$ and continuous in $\overline{\Pi}$. For $z\in \Pi$, we have 
\[
|F(z)|\le \frac{1}{|\Im (z^2)|}\|u\|_{L^2(M)}\|v\|_{L^2(M)}\le \frac{1}{\delta^2}\|u\|_{L^2(M)}\|v\|_{L^2(M)}.
\]
For $z\in\p \Pi$, we also have with a constant $C>0$, 
\begin{equation}
\label{eq_PLMP}
|F(z)|\le C\|u\|_{L^{\frac{2n}{n+2}}(M)}\|v\|_{L^{\frac{2n}{n+2}}(M)},
\end{equation}
and hence, the Phragm\'en-Lindel\"of principle gives us the estimate \eqref{eq_PLMP} for all $z\in \Pi$. Now we get
\begin{align*}
\| (-\Delta_g-z^2)^{-1} u\|_{L^{\frac{2n}{n-2}}(M)}&=\sup_{v\in C^\infty(M),v\ne 0} \frac{|( (-\Delta_g-z^2)^{-1} u, v)_{L^2}|}{\|v\|_{L^{\frac{2n}{n+2}}(M)}}\\
&\le C\|u\|_{L^{\frac{2n}{n+2}}(M)}.
\end{align*}
This completes the proof of the uniform resolvent estimate \eqref{eq_resolvent_est_laplacian_z}  in the sector $\Pi$, and hence, the proof of Theorem \ref{thm_main}, once the the uniform estimate in the crucial region has been established.

\subsection{Analysis in the crucial spectral region} Let us assume here that $\Im z=\delta>0$ and $|\text{Re}\,z|$ is large. Then it will be convenient to  make a semiclassical reduction in \eqref{eq_2_1} so that we write 
\[
|\text{Re}\, z|=\frac{1}{h},
\]
where $0<h\ll 1$ is a semiclassical parameter.  It follows from \eqref{eq_2_1} that 
\begin{equation}
\label{eq_2_2}
(-h^2\Delta_g -1) u =h^2f+ h \alpha u, \quad \alpha:=\pm 2i\delta-h\delta^2.
\end{equation}

First multiplying \eqref{eq_2_2} by $\overline{u}$, integrating over the manifold, taking  the imaginary part,  and using 
H\"older's inequality, we get the following a priori estimate,
\begin{equation}
\label{eq_2_3_0}
\|u\|^2_{L^2(M)}\le \frac{h}{2\delta} \|f\|_{L^{\frac{2n}{n+2}}(M)}\|u\|_{L^{\frac{2n}{n-2}}(M)}.
\end{equation}
Using the Peter-Paul inequality
\[
ab\le \frac{a^2}{2\varepsilon}+\frac{\varepsilon b^2}{2}\le \bigg(\frac{a}{\sqrt{2\varepsilon}}+\frac{\sqrt{\varepsilon} b}{\sqrt{2}}\bigg)^2, \quad a,b\ge 0, \quad \varepsilon >0,
\]
we obtain from \eqref{eq_2_3_0} the following estimate, 
\begin{equation}
\label{eq_2_3}
\|u\|_{L^2(M)}\le \frac{h^{\frac{1}{2}}}{\sqrt{2\delta}}\bigg( \frac{\|f\|_{L^{\frac{2n}{n+2}}(M)}}{\sqrt{2\varepsilon}}+\frac{\sqrt{\varepsilon}  \|u\|_{L^{\frac{2n}{n-2}}(M)} }{\sqrt{2}} \bigg), \quad \varepsilon>0,
\end{equation}
needed in the sequel.

The semiclassical principal symbol of the operator $-h^2\Delta_g -1$ is given by $p_0(x,\xi)=|\xi|_g^2-1\in C^\infty(T^*M)$. Set
\[
\Sigma:=\{(x,\xi)\in T^*M: p_0(x,\xi)=0\}.
\]
 We have that 
\begin{equation}
\label{eq_2_4}
\p_\xi p_0(x,\xi)\ne 0 \quad \text{for all}\quad (x,\xi)\in \Sigma\subset T^*M.
\end{equation}
Then for each $x\in M$,  the hypersurface
\[
\Sigma_x:=\{ \xi\in T^*_xM : p_0(x,\xi)=0\}
\]
is the unit sphere and hence, 
\begin{equation}
\label{eq_2_5}
\Sigma_x \text{ has a nondegenerate second fundamental form at each point $\xi$}.
\end{equation}  

We shall now follow an argument of  \cite{BGT} (see also \cite{Koch_Tataru_Zworski}). 
Let $(x_0,\xi_0)\in \Sigma$. In view of  \eqref{eq_2_4}, we can assume that $\p_{\xi_1} p_0(x_0,\xi_0)\ne 0$.    Thus,  the implicit function theorem implies that there is a neighborhood of $(x_0,\xi_0)$ in $T^*M$ such that the following factorization
\begin{equation}
\label{eq_2_6}
p_0(x,\xi)=e(x,\xi)(\xi_1-a(x,\xi'))
\end{equation}
holds in this neighborhood. Here $\xi=(\xi_1,\xi')$, $a(x,\xi')$, $e(x,\xi)$ are real-valued and $e(x_0,\xi_0)\ne 0$. Then in a possibly smaller neighborhood of $(x_0,\xi_0)$, we have 
\begin{equation}
\label{eq_2_7}
|e(x,\xi)|\ge e_0>0
\end{equation}
for some $e_0$ fixed.   Furthermore, it follows from \eqref{eq_2_5} that 
\begin{equation}
\label{eq_2_8}
\p^2_{\xi'} a(x_0,\xi_0') \text{ is nondegenerate}.
\end{equation}

Since the energy surface $\Sigma$ is compact in $T^*M$, there are a finite number of points $(x^{(j)},\xi^{(j)})\in \Sigma$ and their neighborhoods $V_{(x^{(j)},\xi^{(j)})}$ in $T^*M$ such that 
\begin{equation}
\label{eq_2_9}
\Sigma\subset \cup_{j=1}^N V_{(x^{(j)},\xi^{(j)})},
\end{equation}
and the conditions \eqref{eq_2_6}, \eqref{eq_2_7} and \eqref{eq_2_8} hold in each $V_{(x^{(j)},\xi^{(j)})}$, with $\xi_1$ possibly replaced by some other $\xi_j$ in the factorization \eqref{eq_2_6}.

Let $\rho_j\in C^\infty_0(V_{(x^{(j)},\xi^{(j)})}; [0,1])$, $1\le j\le N$, be a partition of unity, associated to the cover \eqref{eq_2_9}, so that 
\[
\sum_{j=1}^N\rho_j=1 \quad \text{near} \quad \Sigma.
\]
Let $\chi\in C^\infty_0(T^*M; [0,1])$ be such that $\supp(\chi)$ is contained in a neighborhood of $\Sigma$, with $\chi=1$ in a smaller neighborhood of $\Sigma$.  We have
\begin{equation}
\label{eq_2_10_0}
\chi=\sum_{j=1}^N \chi_j,\quad \chi_j:=\rho_j\chi.
\end{equation}

Associated to $\chi_j$ is the corresponding $h$-pseudodifferential operator  $\text{Op}_h^w(\chi_j)\in \Psi^{-\infty}(M)=\cap_N \Psi^N(M)$. Then 
\[
WF_h(\text{Op}_h^w(\chi_j))=\supp(\chi_j). 
\]
Here $WF_h$ stands for the semiclassical wave front set, see \cite{Zworski_book}.

Let $\kappa_i:U_i\to V_i\subset \R^n$ be a set of local charts in $M$ so that $U_i\subset M$ is open and $\cup_{i=1}^L U_i=M$. Let $\varphi_i\in C^\infty_0(U_i)$, $1\le i\le L$, be a partition of unity, associated to the cover $U_i$ so that 
\[
\sum_{i=1}^L\varphi_i=1 \quad \textrm{on}\quad M. 
\] 

Using \eqref{eq_2_2}, we get
\[
(-h^2\Delta_g-1)\varphi_i u=h^2\varphi_i f+h\alpha \varphi_i u+[-h^2\Delta_g,\varphi_i]u.
\]
It follows that 
\begin{equation}
\label{eq_2_11}
(-h^2\Delta_g-1)\text{Op}_h^w(\chi_j)\varphi_i u=h^2\text{Op}_h^w(\chi_j) \varphi_i f+h B_j u, 
\end{equation}
where 
\begin{equation}
\label{eq_B}
\begin{aligned}
B_j=&\alpha \text{Op}_h^w(\chi_j)\varphi_i+ \frac{1}{h} \text{Op}_h^w(\chi_j) [-h^2\Delta_g,\varphi_i]\\
&+ \frac{1}{h}[-h^2\Delta_g,\text{Op}_h^w(\chi_j)]\varphi_i \in \Psi^{-\infty}(M),
\end{aligned}
\end{equation} 
and 
\[
WF_h(B_j)\subset WF_h(\text{Op}_h^w(\chi_j))\subset\subset V_{(x^{(j)},\xi^{(j)})}.
\]
Let $\psi_i, \tilde \varphi_i\in C^\infty_0(U_i)$ be such that $\tilde\varphi_i=1$ near $\supp(\varphi_i)$ and  $\psi_i=1$ near $\supp (\tilde \varphi_i)$. Then  we have, see \cite[Chapter 14]{Zworski_book},
\begin{equation}
\label{eq_2_11_def_on_R}
\psi_i\text{Op}_h^w(\chi_j)\varphi_i=\psi_i\kappa_i^* \tilde{\chi}_{j,i}^w(y,hD_{y};h) (\kappa_i^{-1})^*\varphi_i, 
\end{equation}
where $\tilde{\chi}_{j,i}^w(y,hD_{y};h)$ is the Weyl pseudodifferential operator on $\R^n$ with the symbol $\tilde{\chi}_{j,i}\in S^{-N}(\R^{2n})$ for any $N$, and
the principal symbol 
\begin{equation}
\label{eq_principal_symbol_m-r}
\sigma\big(\tilde{\chi}_{j,i}^w(y,hD_{y};h)  \big) (y,\eta)=\chi_j(\kappa_i^{-1}(y), (\p\kappa_i(\kappa_i^{-1}(y)))^T\eta).
\end{equation}
Furthermore, 
\[
WF_h(\tilde{\chi}_{j,i}^w(y,hD_{y};h))=\supp(\chi_j(\kappa_i^{-1}(y), (\p\kappa_i(\kappa_i^{-1}(y)))^T\eta)).
\] 
We also have
\begin{equation}
\label{eq_pseu_rest}
(1-\psi_i)\text{Op}_h^w(\chi_j)\varphi_i=\mathcal{O}(h^\infty): H^{-N}(M)\to H^N(M),
\end{equation}
for all $N$, see   \cite[Chapter 14]{Zworski_book}.

Using \eqref{eq_2_11_def_on_R}, \eqref{eq_pseu_rest} and \eqref{eq_2_1}, we obtain from  \eqref{eq_2_11} and \eqref{eq_B} that 
\begin{equation}
\label{eq_R_n_main}
\begin{aligned}
(-h^2\Delta_g-1)\psi_i\kappa_i^*\tilde{\chi}_{j,i}^w(y,hD_{y};h) (\kappa_i^{-1})^*\varphi_i u=h^2 \psi_i\kappa_i^*\tilde{\chi}_{j,i}^w(y,hD_{y};h) (\kappa_i^{-1})^*\varphi_i f\\
+h\psi_i\kappa_i^*\tilde{B}_{j,i} (\kappa_i^{-1})^*\tilde \varphi_i u+R_0u,
\end{aligned}
\end{equation}
where $\tilde{B}_{j,i}\in \text{Op}_h^w\big(S^{-N}(\R^{2n})\big)$ for all $N$,  and $R_0=\mathcal{O}(h^\infty): H^{-N}(M)\to H^N(M)$ for all $N$. 

Committing an error $\mathcal{O}(h^\infty)$, in what follows we shall view, as we may,  \eqref{eq_R_n_main} as an equation on $\R^n$. We shall therefore be concerned with the following situation. Let $\varphi,\tilde\varphi,\psi\in C^\infty_0(\R^n)$ be such that $\tilde \varphi=1$ near $\supp(\varphi)$, and $\psi=1$ near $\supp(\tilde\varphi)$. Let $g$ be a $C^\infty$ Riemannian metric on $\R^n$ such that $\p_{x}^\alpha g\in L^\infty$ for all $\alpha$, and let $p_0(x,\xi)=\sum_{i,j=1}^ng^{ij}(x)\xi_i\xi_j-1$. Let $\chi_0\in C^\infty_0(\R^{2n})$ be such that $\supp(\chi_0)$ is in a neighborhood of $(x_0,\xi_0)\in p_0^{-1}(0)$ and such that near $\supp(\chi_0)$, we have
\begin{equation}
\label{eq_2_6_new}
p_0(x,\xi)=e(x,\xi) (\xi_1-a(x,\xi')), 
\end{equation}
where  $\xi=(\xi_1,\xi')$ and $e$ satisfies 
\begin{equation}
\label{eq_2_7_R_n}
|e(x,\xi)|\ge e_0>0. 
\end{equation}
Furthermore, 
\begin{equation}
\label{eq_nondegenerate_a}
\p_{\xi'}^2 a(x,\xi')\text{ is nondegenerate near }\supp(\chi_0), 
\end{equation}
see \eqref{eq_2_8}.   Let $B\in \text{Op}_h^w\big( S^{-N}(\R^{2n})\big)$ for all $N$, and let $\chi=\chi(x,\xi;h)\in  S^{-N}(\R^{2n})$ for all $N$ be such that 
\[
\chi=\chi_0+\mathcal{O}(h)\quad \text{in}\quad S^{-N}(\R^{2n}),
\]
and 
\begin{equation}
\label{eq_wavefront_set}
WF_h(\chi^w(x,hD_x;h))=\supp(\chi_0).
\end{equation} 
When $u\in L^2(\R^n)\cap C^\infty(\R^n)$, consider the equation,
\begin{equation}
\label{eq_R_n_main_new}
(-h^2\Delta_g-1)\psi\chi^w(x,hD_x;h)\varphi u=h^2\psi\chi^w(x,hD_x;h)\varphi f+h\psi B\tilde \varphi u+R_0u,
\end{equation}
where 
\begin{equation}
\label{eq_non_local_R_0}
\|R_0u\|_{L^2(\R^n)}=\mathcal{O}(h^\infty)\|u\|_{L^2(\R^n)}.
\end{equation} 

Let us now extend $e(x,\xi)$ arbitrarily to a symbol in $S^0(\R^{2n})$ with the bound \eqref{eq_2_7_R_n} in all of $\R^{2n}$, and  let us also extend $a(x_1,x',\xi')$ to a real-valued element of $C^\infty_0(\R_{x_1}, S^{0}(\R^{2(n-1)}))$.

Microlocal factorization \eqref{eq_2_6_new} yields that 
\begin{align*}
(-h^2\Delta_g-1)\chi^w(x,hD_x;h)=&e^w(x,hD_x)(hD_{x_1}-a^w(x,hD_{x'}))\chi^w(x,hD_x;h)\\
&+hR^w(x,hD_{x};h),
\end{align*}
where $R^w\in \text{Op}_h^w(S^{-N}(\R^{2n}))$ for all $N>0$.  
This implies that 
\begin{equation}
\label{eq_2_12}
\begin{aligned}
(-h^2\Delta_g-1)\psi\chi^w(x,hD_x;h)=&e^w(x,hD_x)(hD_{x_1}-a^w(x,hD_{x'}))\psi \chi^w(x,hD_x;h)\\
&+hR_1^w(x,hD_{x};h),
\end{aligned}
\end{equation}
where $R_1^w(x,hD_{x};h)\in  \text{Op}_h^w(S^{-N}(\R^{2n}))$ for all $N$. 

Since $e\in S^0(\R^{2n})$ is elliptic, see \eqref{eq_2_7_R_n}, there exists $h_0>0$ such that for all $0<h\le h_0$,  the inverse $e^w(x,hD_x)^{-1}$ exists and $e^w(x,hD_x)^{-1}\in \text{Op}^w_h(S^{0}(\R^{2n}))$. Therefore, we conclude from \eqref{eq_R_n_main_new} and \eqref{eq_2_12} that 
\begin{equation}
\label{eq_2_13}
(hD_{x_1}-a^w(x,hD_{x'}))\psi \chi^w(x,hD_x;h)\varphi u=h^2\tilde f+h B_1\tilde \varphi u + e^w(x,hD_x)^{-1}R_0u, 
\end{equation}
where
\begin{equation}
\label{eq_tilde_f}
\tilde f=e^w(x,hD_{x})^{-1}\psi  \chi^w(x,hD_x;h)\varphi f,
\end{equation}
\begin{equation}
\label{eq_tilde_B}
B_1=e^w(x,hD_x)^{-1} (\psi B-R_1^w(x,hD_{x};h)\varphi ).
\end{equation}
Now if  $\supp(\psi)\cap\pi_x(\supp(\chi_0))=\emptyset$, then using  \eqref{eq_wavefront_set}, we get 
\[
\psi(x)\chi^w(x,hD_x;h)=\mathcal{O}_{\mathcal{S}'(\R^n)\to \mathcal{S}(\R^n)}(h^\infty).
\]
Here $\pi_x: (x,\xi)\mapsto x$ is the projection. 
Modifying  $B_1\in \text{Op}_h^w(S^0(\R^{2n}))$ in \eqref{eq_2_13},  we can assume therefore that $\supp(\psi)$ is contained in a small neighborhood of $x_0$. 

Let $x_0=(x_{0,1},x_0')\in \R\times\R^{n-1}$ and let $(l_1,l_2)$ be an interval around   $x_{0,1}$, close to $x_{0,1}$,  so that $\pi_{x_1}(\supp (\psi))\subset (l_1,l_2)$.  By Duhamel's formula \eqref{eq_4_4} applied to \eqref{eq_2_13}, we get for $x_1\in (l_1,l_2)$,
\begin{equation}
\label{eq_2_14}
\begin{aligned}
(\psi & \chi^w (x,hD_x;h)\varphi u)(x_1,x')=i \int_{l_1}^{x_1} F(x_1,s)  (B_1\tilde \varphi u)(s,x')ds\\
&+ \frac{i}{h} \int_{l_1}^{x_1} F(x_1,s)  ((e^{w})^{-1}R_0 u)(s,x')ds +ih\int_{l_1}^{x_1} F(x_1,r)F(s,r)^*\tilde f(s,x')ds,
\end{aligned}
\end{equation}
 for all fixed $r\in \R$.
Here $\{F(x_1,r)\}_{x_1\in \R}$ is a family of unitary operators on $L^2(\R^{n-1})$ solving 
\begin{align*}
(hD_{x_1}-a^w(x,hD_{x'})) F(x_1,r)&=0,\\
F(r,r)&=I,
\end{align*}
for $x\in \R^n$ and for all $r\in \R$. The unitarity of $F(x_1,r)$ is a consequence of the fact that  $a^w(x,hD_{x'})$ is self-adjoint.

Let $0\le \tilde \chi\in C^\infty_0(T^*\R^{n-1})$ be  such that  $\tilde \chi=1$ near $\pi_{(x',\xi')}(\supp (\chi_0))$ and $\supp(\tilde \chi)$ is in a small neighborhood of $(x_0',\xi_0')$ so that the condition \eqref{eq_nondegenerate_a} holds on 
$\supp(\tilde \chi)$ for all $x_1$ in a neighborhood of $x_{0,1}$. Here 
\[
\pi_{(x',\xi')}:T^*\R^n\to T^*\R^{n-1},\quad (x_1,x',\xi_1,\xi')\mapsto (x',\xi').
\]
Hence, by the  composition formula of $h$-pseudodifferential operators \cite[Theorem 4.11]{Zworski_book}, see also \cite[Proposition 9.5]{Dimassi_Sjostrand_book}, we get
\begin{equation}
\label{eq_2_15}
(1- \tilde \chi^w(x',hD_{x'}))\psi \chi^w(x,hD_x;h)=\mathcal{O}_{\mathcal{S}'(\R^n)\to\mathcal{S}(\R^n)}(h^\infty). 
\end{equation}

Using \eqref{eq_2_15} and \eqref{eq_2_14}, we get  
\begin{equation}
\label{eq_2_16}
\begin{aligned}
(\psi & \chi^w (x,hD_x;h)\varphi u)(x_1,x')=i \int_{l_1}^{x_1} \tilde \chi^w(x',hD_{x'}) F(x_1,s)  (B_1\tilde \varphi u)(s,x')ds\\
&+ \frac{i}{h} \int_{l_1}^{x_1} \tilde \chi^w(x',hD_{x'}) F(x_1,s)  ((e^{w})^{-1}R_0 u)(s,x')ds\\
& +ih\int_{l_1}^{x_1} \tilde \chi^w(x',hD_{x'}) F(x_1,r)F(s,r)^*\tilde f(s,x')ds
+\mathcal{O}_{\mathcal{S}(\R^n)}(h^\infty)\|\varphi u\|_{L^2(\R^n)}.
\end{aligned}
\end{equation}

We shall next estimate $\|\psi \chi^w (x,hD_x;h)\varphi u\|_{L^{\frac{2n}{n-2}}(\R^n)}$.  First, a repeated application of Lemma \ref{lem_NSZ}, allows us to write 
\begin{align*}
\psi \chi^w (x,hD_x;h)\varphi u=&\rho^w(x_1,hD_{x_1}) \psi \chi^w (x,hD_x;h)\varphi u\\
&+\mathcal{O}_{\mathcal{S}(\R^n)}(h^\infty)\|\psi \chi^w (x,hD_x;h)\varphi u\|_{L^2(\R^n)},
\end{align*}
where $\rho\in C^\infty_0(T^*\R)$.  In view of Lemma \ref{lem_Zworski_lem_2_2}, we get  
\begin{equation}
\label{eq_2_16_2}
\begin{aligned}
\|\psi \chi^w (x,hD_x;h)\varphi u\|_{L^{\frac{2n}{n-2}}(\R^n)}\le &C h^{-\frac{1}{2n}} \|\psi \chi^w (x,hD_x;h)\varphi u\|_{L^{\frac{2n}{n-1}}_{x_1}L^{\frac{2n}{n-2}}_{x'}}\\
&+\mathcal{O}(h^\infty)\|\varphi u\|_{L^2(\R^n)}.
\end{aligned}
\end{equation}

Let us now proceed to estimate $\|\psi \chi^w (x,hD_x;h)\varphi u\|_{L^{\frac{2n}{n-1}}_{x_1}L^{\frac{2n}{n-2}}_{x'}}$ using the semiclassical Strichartz estimates \eqref{eq_3_4} and \eqref{eq_3_5}.  To that end, we start by estimating the first integral  in the right hand side of \eqref{eq_2_16},  
\[
J_1(x_1,x'):=i \int_{l_1}^{x_1} \tilde \chi^w(x',hD_{x'}) F(x_1,s)  (B_1\tilde \varphi u)(s,x')ds,\quad x_1\in (l_1,l_2).
\]
 Letting  
\begin{equation}
\label{eq_U}
U(x_1,s)=\textbf{1}_{[0,l_2-l_1]}(x_1-s)\tilde \chi^w(x',hD_{x'}) F(x_1,s),
\end{equation}
we have 
\[
J_1(x_1,x')=i \int_{l_1}^{x_1} U(x_1,s)   (B_1\tilde \varphi u)(s,x')  ds.
\]
Following \cite[the proof of Theorem 10.8]{Zworski_book}, we write
\begin{align*}
J_1(x_1,x')=i\int_{\R} \textbf{1}_{(l_1,l_2)}(s) \textbf{1}_{(-\infty,x_1)}(s)U(x_1,s)  (B_1\tilde \varphi u)(s,x')ds\\
=i\int_{\R} \textbf{1}_{(l_1,l_2)}(s) \textbf{1}_{(s,+\infty)}(x_1)U(x_1,s) (B_1\tilde \varphi u)(s,x')ds.
\end{align*}
Letting $q=\frac{2n}{n-2}$ and $p=\frac{2n}{n-1}$ and using Minkowski's inequality, we get 
\begin{align*}
\| J_1\|_{L^p_{x_1}L^q_{x'}} &\le \int_{\R} \|\textbf{1}_{(l_1,l_2)}(s) \textbf{1}_{(s,+\infty)}(x_1)U(x_1,s)  (B_1\tilde \varphi u)(s,x')\|_{L^p_{x_1}L^q_{x'}}ds\\
&\le \int_{l_1}^{l_2} \|U(x_1,s)(B_1\tilde \varphi u)(s,x')\|_{L^p_{x_1}L^q_{x'}}ds
\end{align*}

In view of the condition \eqref{eq_nondegenerate_a}, using the semiclassical Strichartz estimate \eqref{eq_3_4} with $k=n-1$, $q=\frac{2n}{n-2}$ and $p=\frac{2n}{n-1}$ in the last term of the above estimate, we obtain that 
\begin{equation}
\label{eq_2_16_3}
\begin{aligned}
\| J_1\|_{L^p_{x_1}L^q_{x'}}&\le Ch^{-1/p} \int_{l_1}^{l_2} \| (B_1\tilde \varphi u)(s,x')\|_{L^2_{x'}} ds\le Ch^{-1/p} \|  B_1\tilde \varphi u\|_{L^2_{x}}\\
&\le Ch^{-1/p} \|\tilde \varphi  u\|_{L^2_{x}}. 
\end{aligned}
\end{equation}
Here we have used the fact that the operator $B_1$ in \eqref{eq_2_13} is  bounded on $L^2(\R^n)$ uniformly in $h$ for all $h>0$ small enough. 

Similarly, we can estimate the second integral in the right hand side of  \eqref{eq_2_16},  
\[
J_2(x_1,x'):=\frac{i}{h} \int_{l_1}^{x_1} \tilde \chi^w(x',hD_{x'}) F(x_1,s)  ((e^{w})^{-1}R_0 u)(s,x')ds,
\]
obtaining the bound
\begin{equation}
\label{eq_J_2_add}
\| J_2\|_{L^p_{x_1}L^q_{x'}}\le \mathcal{O}(h^\infty)\|u\|_{L^2(\R^n)}.
\end{equation}
Here we have also used \eqref{eq_non_local_R_0}.

Let us now estimate the third integral in the right hand side \eqref{eq_2_16},  
\[
J_3(x_1,x'):=ih\int_{l_1}^{x_1} \tilde \chi^w(x',hD_{x'}) F(x_1,r)F(s,r)^*\tilde f(s,x')ds,
\]
 where $\tilde f$ is given by \eqref{eq_tilde_f}.  To that end we write
 \begin{equation}
 \label{eq_J_2}
 J_3(x_1,x')= J_{3,1}(x_1,x')+J_{3,2}(x_1,x'),
 \end{equation}
 where 
 \begin{equation}
 \label{eq_J_2_new_revised}
 \begin{aligned}
 J_{3,1}(x_1,x')=ih\int_{l_1}^{x_1} \tilde \chi^w(x',hD_{x'}) F(x_1,r)F(s,r)^* \tilde \chi^w(x',hD_{x'}) \tilde  f(s,x')ds, \\
 J_{3,2}(x_1,x')=ih\int_{l_1}^{x_1} \tilde \chi^w(x',hD_{x'}) F(x_1,s) (1- \tilde \chi^w(x',hD_{x'}))\tilde f(s,x')ds.
 \end{aligned}
 \end{equation}
 In view of \eqref{eq_2_15},
 \begin{equation}
 \label{eq_J_3_2_help}
 \begin{aligned}
 (1- &\tilde \chi^w(x',hD_{x'}))\tilde f=(1- \tilde \chi^w(x',hD_{x'}))  e^w(x,hD_x)^{-1}\psi \chi^w(x,hD_x;h)\varphi f\\
 &=h^{-2} (1- \tilde \chi^w(x',hD_{x'}))  e^w(x,hD_x)^{-1}\psi \chi^w(x,hD_x;h)\varphi (-h^2\Delta_g-(hz)^2)\tilde \varphi u\\
 &=\mathcal{O}_{\mathcal{S}(\R^{n})}(h^\infty)\|\tilde \varphi u\|_{L^2(\R^n)}.
 \end{aligned}
 \end{equation}
 Now writing 
 \[
 J_{3,2}(x_1,x')=ih\int_{l_1}^{x_1} U(x_1,s) (1- \tilde \chi^w(x',hD_{x'}))\tilde f(s,x')ds,
 \]
 where $U(x_1,s)$ is given by \eqref{eq_U}, and estimating $J_{3,2}$ similarly to $J_1$ above, using \eqref{eq_J_3_2_help}, we get 
\begin{equation}
\label{eq_J_2_2}
\|J_{3,2}\|_{L^{\frac{2n}{n-1}}_{x_1} L^{\frac{2n}{n-2}}_{x'}}=\mathcal{O}(h^\infty)\|\tilde \varphi u\|_{L^2(\R^n)}.
\end{equation}

We shall next estimate $J_{3,1}$. In doing so, we let $r=l_1$. 
Since $\tilde \chi (x',\xi')$ is real-valued, its Weyl quantization  $\chi^w(x',hD_{x'})$ is self-adjoint, and therefore, the $L^2$ adjoint of $U(s,l_1)$,  given by \eqref{eq_U}, is as follows,
 \[
 U(s,l_1)^*= \textbf{1}_{[0,l_2-l_1]}(s-l_1)F(s,l_1)^*\tilde \chi^w(x',hD_{x'}).
 \]  
Hence,  for $r=l_1$, we get
\[
  J_{3,1}(x_1,x')=ih\int_{-\infty}^{x_1} U(x_1,l_1)U(s,l_1)^*\tilde f(s,x')ds, \quad x_1\in (l_1,l_2).
 \]

By the semiclassical Strichartz estimate \eqref{eq_3_5} with $k=n-1$, $q=\tilde q=\frac{2n}{n-2}$ and $p=\tilde p=\frac{2n}{n-1}$, we obtain that 
\begin{equation}
\label{eq_2_17}
\|J_{3,1}\|_{L^{\frac{2n}{n-1}}_{x_1} L^{\frac{2n}{n-2}}_{x'}}\le C h^{\frac{1}{n}}\|\tilde f\|_{L^{\frac{2n}{n+1}}_{x_1} L^{\frac{2n}{n+2}}_{x'}}.
\end{equation}
Since $\tilde f$ is microlocalized in a compact subset of $T^*\R^n$, a repeated application of Lemma \ref{lem_NSZ} shows that 
\[
\tilde f=\rho^w(x_1,hD_{x_1}) \tilde f+\mathcal{O}_{\mathcal{S}(\R^n)}(h^\infty)\|\tilde f\|_{L^2(\R^n)},
\]
where $\rho\in C^\infty_0(T^*\R)$.  By Lemma \ref{lem_Zworski_lem_2_2} and the fact that 
\[
\tilde f=h^{-2}e^w(x,hD_{x})^{-1}\psi \chi^w(x,hD_x;h)\varphi (-h^2\Delta_g-(hz)^2)\tilde \varphi u,
\]
see \eqref{eq_tilde_f},
we get 
\begin{equation}
\label{eq_2_18}
\|\tilde f \|_{L^{\frac{2n}{n+1}}_{x_1} L^{\frac{2n}{n+2}}_{x'}}\le C h^{-\frac{1}{2n}}\|\tilde f\|_{L^{\frac{2n}{n+2}}(\R^n)}
+\mathcal{O}(h^\infty)\|\tilde \varphi u\|_{L^2(\R^n)},
\end{equation}
for all $0<h$ small enough. Using \eqref{eq_J_2}, \eqref{eq_2_17}, \eqref{eq_2_18},  \eqref{eq_J_2_2},  \eqref{eq_tilde_f}, and  the fact that 
$e^w(x,hD_{x})^{-1}\psi \chi^w(x,hD_x;h)\in \text{Op}_h^w(S^0(\R^{2n}))$
 is bounded on $L^p$ uniformly in $h$, for all $0<h$ small enough, see 
  \cite[Theorem 2.1]{Taylor_book_pseudo},
 we conclude that 
\begin{equation}
\label{eq_2_19}
\|J_3\|_{L^{\frac{2n}{n-1}}_{x_1} L^{\frac{2n}{n-2}}_{x'}}\le C h^{\frac{1}{2n}}\|    \varphi   f\|_{L^{\frac{2n}{n+2}}(\R^n)}+\mathcal{O}(h^\infty)\|\tilde \varphi u\|_{L^2(\R^n)}.
\end{equation}

By \eqref{eq_2_16}, \eqref{eq_2_16_3}, \eqref{eq_J_2_add} and \eqref{eq_2_19}, we get
\begin{align*}
 \|\psi \chi^w(x,hD_x;h)\varphi u\|_{L^{\frac{2n}{n-1}}_{x_1}L^{\frac{2n}{n-2}}_{x'}}\le& C h^{-\frac{1}{2}+\frac{1}{2n}} \|\tilde \varphi u\|_{L^2(\R^n)} +Ch^{\frac{1}{2n}}\|\varphi f\|_{L^{\frac{2n}{n+2}}(\R^n)}\\
 &+\mathcal{O}(h^\infty)\|u\|_{L^2(\R^n)},
\end{align*}
and therefore, using 
 \eqref{eq_2_16_2}, we obtain that 
\begin{equation}
\label{eq_2_20}
\begin{aligned}
\|\psi \chi^w(x,hD_x;h)\varphi u\|_{L^{\frac{2n}{n-2}}(\R^n)}\le &C h^{-\frac{1}{2}} \|\tilde \varphi u\|_{L^2(\R^n)}+C\|\varphi f\|_{L^{\frac{2n}{n+2}}(\R^n)}\\
&+\mathcal{O}(h^\infty)\|u\|_{L^2(\R^n)}.
\end{aligned}
\end{equation}

We shall return to the compact manifold $M$. In view of  \eqref{eq_2_11_def_on_R} and \eqref{eq_2_20}, we have 
\begin{equation}
\label{eq_2_20_2}
\begin{aligned}
\|\psi_i \text{Op}^w_h(\chi_j)\varphi_i u\|_{L^{\frac{2n}{n-2}}(M)}\le & C h^{-\frac{1}{2}} \| \tilde \varphi_i u\|_{L^2(M)}+C\| f\|_{L^{\frac{2n}{n+2}}(M)}\\
&+\mathcal{O}(h^\infty)\|u\|_{L^2(M)},
\end{aligned}
\end{equation}
for $i=1, \dots, L$ and $j=1,\dots, N$. By \eqref{eq_pseu_rest}
 and Sobolev's embedding, we get 
\[
(1-\psi_i)\text{Op}^w_h(\chi_j)\varphi_i =\mathcal{O}(h^\infty): L^2(M)\to L^{\frac{2n}{n-2}}(M),
\]
and therefore, summing over $i$, we obtain  that 
\[
\| \text{Op}^w_h(\chi_j)u\|_{L^{\frac{2n}{n-2}}(M)}\le C h^{-\frac{1}{2}} \| u\|_{L^2(M)}+C\| f\|_{L^{\frac{2n}{n+2}}(M)}.
\]
Hence, 
\begin{equation}
\label{eq_2_21}
\| \text{Op}^w_h(\chi)u\|_{L^{\frac{2n}{n-2}}(M)}\le C h^{-\frac{1}{2}} \| u\|_{L^2(M)}+C\| f\|_{L^{\frac{2n}{n+2}}(M)}.
\end{equation}

Let us now estimate 
$ \| (1- \text{Op}^w_h(\chi) )u\|_{L^{\frac{2n}{n-2}}(M)}$. 
To that end, using the equation \eqref{eq_2_2}, we get 
\begin{equation}
\label{eq_2_23}
(-h^2\Delta_g-1)(1- \text{Op}^w_h(\chi))u=h^2(1-\text{Op}^w_h(\chi)) f+h(1-\text{Op}^w_h(\chi))\alpha u+hLu,
\end{equation}
where $L=h^{-1}[h^2\Delta_g, \text{Op}^w_h(\chi)]\in \Psi^{-\infty}(M)$. 

For the semiclassical principal symbol $p_0(x,\xi)=|\xi|_g^2-1$ of the operator $-h^2\Delta_g-1$, we have $p_0(x,\xi)\ne 0$ on $\supp(1-\chi)$, and therefore, 
\[
|p_0(x,\xi)|\ge \frac{\langle \xi\rangle_g^2}{C}, \quad \langle \xi\rangle_g=\sqrt{1+|\xi|_g^2},
\]
for all $(x,\xi)\in \supp(1-\chi)$, i.e. the operator $-h^2\Delta_g-1$ is elliptic on $\supp(1-\chi)$. 
Hence,  there exists an operator $E\in \text{Op}^w_h(S^{-2}(T^*M))$ such that 
\begin{equation}
\label{eq_2_24}
E(-h^2\Delta_g-1)(1-\text{Op}^w_h(\chi))=1-\text{Op}^w_h(\chi)+R,
\end{equation}
where 
\begin{equation}
\label{eq_2_25}
R\in \cap_{N\ge 0, M\ge 0} h^N\text{Op}^w_h(S^{-M}).
\end{equation} 

Applying the operator $E$ to \eqref{eq_2_23} and using \eqref{eq_2_24}, we get 
\begin{equation}
\label{eq_2_26}
\begin{aligned}
(1-\text{Op}^w_h(\chi))u=&-Ru+h^2E(1-\text{Op}^w_h(\chi))f+hE(1-\text{Op}^w_h(\chi))\alpha u\\
&+hELu.
\end{aligned}
\end{equation}
It follows from \eqref{eq_2_25}  that 
\begin{equation}
\label{eq_2_27}
\|Ru\|_{L^{\frac{2n}{n-2}}(M)}=\mathcal{O}(h^\infty)\|u\|_{L^{\frac{2n}{n-2}}(M)},
\end{equation}
see \cite[Theorem 2.2]{Taylor_book_pseudo}.
As $E\in \text{Op}^w_h(S^{-2}(T^*M))$,  we have $E:L^{\frac{2n}{n-2}}(M)\to L^{\frac{2n}{n-2}}(M)$ is bounded uniformly for $0<h\ll 1$, see \cite[Theorem 2.2]{Taylor_book_pseudo},  and therefore, 
\begin{equation}
\label{eq_2_28}
\|hE(1-\text{Op}^w_h(\chi))\alpha u+hELu\|_{L^{\frac{2n}{n-2}}(M)}=\mathcal{O}(h)\|u\|_{L^{\frac{2n}{n-2}}(M)}.
\end{equation}
Furthermore, as 
\[
E: W^{-2,\frac{2n}{n-2}}_{\text{scl}}(M)\to L^{\frac{2n}{n-2}}(M)
\]
is uniformly bounded for $0<h\ll 1$, see \cite[Theorem 2.2]{Taylor_book_pseudo}, and from the semiclassical Sobolev embedding \eqref{eq_semiclass_emb_dual}, we obtain that 
\begin{equation}
\label{eq_2_29}
\begin{aligned}
h^2\| E(1- \text{Op}^w_h(\chi) )f\|_{L^{\frac{2n}{n-2}}(M)}&\le Ch^2\|(1- \text{Op}^w_h(\chi) )f\|_{W^{-2,\frac{2n}{n-2}}_{\text{scl}}(M)}\\
&\le C\|f\|_{L^{\frac{2n}{n+2}}(M)}.
\end{aligned}
\end{equation}
We conclude from \eqref{eq_2_26}, using \eqref{eq_2_27}, \eqref{eq_2_28} and \eqref{eq_2_29} that 
\begin{equation}
\label{eq_2_30}
\| (1- \text{Op}^w_h(\chi) )u\|_{L^{\frac{2n}{n-2}}(M)}\le \mathcal{O}(h)\|u\|_{L^{\frac{2n}{n-2}}(M)} +C\|f\|_{L^{\frac{2n}{n+2}}(M)}.
\end{equation}
It follows from \eqref{eq_2_21} and \eqref{eq_2_30} that 
\begin{equation}
\label{eq_2_30_2damp}
\|u\|_{L^{\frac{2n}{n-2}}(M)}\le C(h^{-\frac{1}{2}}\|u\|_{L^2(M)}+h\|u\|_{L^{\frac{2n}{n-2}}(M)} +\|f\|_{L^{\frac{2n}{n+2}}(M)}).
\end{equation}
By \eqref{eq_2_3}, we get 
\begin{align*}
\|u\|_{L^{\frac{2n}{n-2}}(M)}\le C \bigg( \frac{1}{\sqrt{2\varepsilon}}\|f\|_{L^{\frac{2n}{n+2}}(M)} &+\frac{\sqrt{\varepsilon}}{\sqrt{2}} \|u\|_{L^{\frac{2n}{n-2}}(M)}\\
&+ h\|u\|_{L^{\frac{2n}{n-2}}(M)} +\|f\|_{L^{\frac{2n}{n+2}}(M)} \bigg),
\end{align*}
and therefore, taking $\varepsilon>0$ and $h_0>0$ sufficiently small but fixed, we obtain that for all $h\in (0,h_0]$,
\[
\|u\|_{L^{\frac{2n}{n-2}}(M)}\le C \|f\|_{L^{\frac{2n}{n+2}}(M)}.
\]
This completes the proof of the uniform resolvent estimate \eqref{eq_resolvent_est_laplacian_z}  in the region where $\Im z=\delta$ and $|\Re z|\ge C$, for some $C>0$ large enough, and thus, 
the proof of Theorem \ref{thm_main}.

\section{Laplace operator. Proof of Remark \ref{rem_main}}

\label{sec_Laplace_Remark}

Let us first remark that in view of the Riesz--Thorin interpolation theorem,  the estimate \eqref{eq_rem_main} is a consequence of \eqref{eq_resolvent_est_laplacian_z} and the following endpoint bound, 
\begin{equation}
\label{eq_Lap_rev_0}
\|u\|_{L^{\frac{2(n+1)}{n-1}}(M)}\le C |z|^{-\frac{2}{n+1}}\|(-\Delta_g-z^2)u\|_{L^{\frac{2(n+1)}{n+3}}(M)},
\end{equation}
valid for $u\in C^\infty(M)$ and $z\in\Xi_\delta$.  In what follows we shall therefore concentrate on proving \eqref{eq_Lap_rev_0}. 
When doing so, we shall continue to use the same notation as in the previous sections, and we shall proceed by inspection of the arguments used in the proof of Theorem \ref{thm_main}.

\subsection{Analysis in the crucial spectral region} Here we assume that $\Im z=\delta>0$ and $|\Re z|$ is large, and then we see that similarly 
to the a priori estimate \eqref{eq_2_3_0}, we have  
\begin{equation}
\label{eq_Lap_rev_1}
\|u\|^2_{L^2(M)}\le \frac{h}{2\delta} \|f\|_{L^{\frac{2(n+1)}{n+3}}(M)}\|u\|_{L^{\frac{2(n+1)}{n-1}}(M)}.
\end{equation}

Next, we consider the formula \eqref{eq_2_16}, where we now have to estimate the norm $\|\psi \chi^w (x,hD_x;h)\varphi u\|_{L^{\frac{2(n+1)}{n-1}}(\R^n)}$.  Contrary to the proof of Theorem \ref{thm_main}, to that end, we do not need to use the semiclassical embedding of Lemma \ref{lem_Zworski_lem_2_2} but we rely on the semiclassical Strichartz estimates \eqref{eq_3_4} and \eqref{eq_3_5} only.  To bound the first integral in the right hand side of \eqref{eq_2_16}, using \eqref{eq_3_4} with $k=n-1$ and $p=q=\frac{2(n+1)}{n-1}$, we get 
\begin{equation}
\label{eq_Lap_rev_2}
\|J_1\|_{L^{\frac{2(n+1)}{n-1}}(\R^n)}\le C h^{-\frac{(n-1)}{2(n+1)}}\|\tilde \varphi u\|_{L^2(\R^n)},
\end{equation}
cf.  \eqref{eq_2_16_3}.  The estimate for the second integral in the right hand side of  \eqref{eq_2_16} is as follows, 
\begin{equation}
\label{eq_Lap_rev_3}
\| J_2\|_{L^{\frac{2(n+1)}{n-1}}(\R^n)}\le \mathcal{O}(h^\infty)\|u\|_{L^2(\R^n)},
\end{equation}
cf. \eqref{eq_J_2_add}.  To bound  the third integral in the right hand side \eqref{eq_2_16},  we shall estimate $J_{3,1}$ and $J_{3,2}$ given by \eqref{eq_J_2} and \eqref{eq_J_2_new_revised}. First we have 
\begin{equation}
\label{eq_Lap_rev_4}
\|J_{3,2}\|_{L^{\frac{2(n+1)}{n-1}}(\R^n)}=\mathcal{O}(h^\infty)\|\tilde \varphi u\|_{L^2(\R^n)},
\end{equation}
cf. \eqref{eq_J_2_2}.  Using the semiclassical Strichartz estimate \eqref{eq_3_5} with $k=n-1$, $p=q=\tilde p=\tilde q=\frac{2(n+1)}{n-1}$, we get  
\begin{equation}
\label{eq_Lap_rev_5}
\|J_{3,1}\|_{L^{\frac{2(n+1)}{n-1}}(\R^n)}\le C h^{\frac{2}{n+1}}\|\varphi f\|_{L^{\frac{2(n+1)}{n+3}}(\R^n)},
\end{equation}
cf. \eqref{eq_2_17}.  By \eqref{eq_Lap_rev_2}, \eqref{eq_Lap_rev_3}, \eqref{eq_Lap_rev_4}, \eqref{eq_Lap_rev_5}, we obtain that 
\begin{align*}
\|\psi \chi^w (x,hD_x;h)\varphi u\|_{L^{\frac{2(n+1)}{n-1}}(\R^n)}&\le C h^{-\frac{(n-1)}{2(n+1)}}\|\tilde \varphi u\|_{L^2(\R^n)}\\
&+ C h^{\frac{2}{n+1}}\| \varphi f\|_{L^{\frac{2(n+1)}{n+3}}(\R^n)}+\mathcal{O}(h^\infty)\| u\|_{L^2(\R^n)},
\end{align*}
cf. \eqref{eq_2_20}.  Returning to the compact manifold $M$, we conclude that 
\begin{equation}
\label{eq_Lap_rev_6}
\| \text{Op}^w_h(\chi)u\|_{L^{\frac{2(n+1)}{n-1}}(M)}\le C h^{-\frac{(n-1)}{2(n+1)}}\|u\|_{L^2(M)}+ C h^{\frac{2}{n+1}}\|f\|_{L^{\frac{2(n+1)}{n+3}}(M)},
\end{equation}
cf. \eqref{eq_2_21}.

Let us now estimate 
$ \| (1- \text{Op}^w_h(\chi) )u\|_{L^{\frac{2(n+1)}{n-1}}(M)}$. To this end, we shall examine the expression \eqref{eq_2_26}. When doing so we rely on the fact that $E\in \text{Op}^w_h(S^{-2}(T^*M))$,  and thus, for any $s\in [0,2]$, and $1<p<\infty$, we have
\begin{equation}
\label{eq_Lap_rev_7}
E: W^{-s,p}_{\text{scl}}(M)\to L^{p}(M)
\end{equation}
is uniformly bounded for $0<h\ll 1$, see \cite[Theorem 2.5, page 268]{Taylor_book_pseudo},

Using \eqref{eq_Lap_rev_7} with $s=\frac{2n}{n+1}$ and the semiclassical Sobolev embedding \eqref{eq_semiclass_emb_dual},  we obtain that 
\begin{equation}
\label{eq_Lap_rev_8}
\begin{aligned}
h^2\| E(1- \text{Op}^w_h(\chi) )f\|_{L^{\frac{2(n+1)}{n-1}}(M)}&\le Ch^2\|(1- \text{Op}^w_h(\chi) )f\|_{W^{-\frac{2n}{n+1},\frac{2(n+1)}{n-1}}_{\text{scl}}(M)}\\
&\le Ch^{\frac{2}{n+1}}\|f\|_{L^{\frac{2(n+1)}{n+3}}(M)},
\end{aligned}
\end{equation}
cf. \eqref{eq_2_29}. We also have  
\begin{equation}
\label{eq_Lap_rev_9}
\|hE(1-\text{Op}^w_h(\chi))\alpha u+hELu\|_{L^{\frac{2(n+1)}{n-1}}(M)}=\mathcal{O}(h)\|u\|_{L^{\frac{2(n+1)}{n-1}}(M)},
\end{equation}
cf. \eqref{eq_2_28}, and 
\begin{equation}
\label{eq_Lap_rev_10}
\|Ru\|_{L^{\frac{2(n+1)}{n-1}}(M)}=\mathcal{O}(h^\infty)\|u\|_{L^{\frac{2(n+1)}{n-1}}(M)},
\end{equation}
cf. \eqref{eq_2_27}.
We conclude from \eqref{eq_2_26}, using \eqref{eq_Lap_rev_8}, \eqref{eq_Lap_rev_9} and \eqref{eq_Lap_rev_10}  that 
\begin{equation}
\label{eq_Lap_rev_11}
\| (1- \text{Op}^w_h(\chi) )u\|_{L^{\frac{2(n+1)}{n-1}}(M)}\le \mathcal{O}(h)\|u\|_{L^{\frac{2(n+1)}{n-1}}(M)} +Ch^{\frac{2}{n+1}}\|f\|_{L^{\frac{2(n+1)}{n+3}}(M)},
\end{equation}
cf. \eqref{eq_2_30}.

Now it follows from \eqref{eq_Lap_rev_6} and \eqref{eq_Lap_rev_11} that 
\[
\|u\|_{L^{\frac{2(n+1)}{n-1}}(M)}\le C \bigg(h^{-\frac{(n-1)}{2(n+1)}}\|u\|_{L^2(M)}+h\|u\|_{L^{\frac{2(n+1)}{n-1}}(M)}  + h^{\frac{2}{n+1}}\|f\|_{L^{\frac{2(n+1)}{n+3}}(M)}\bigg),
\]
cf. \eqref{eq_2_30_2damp}, and therefore, by \eqref{eq_Lap_rev_1} and the Peter--Paul inequality, we get 
\begin{align*}
\|u\|_{L^{\frac{2(n+1)}{n-1}}(M)}\le C \bigg(h^{\frac{1}{n+1}} \|f\|^{1/2}_{L^{\frac{2(n+1)}{n+3}}(M)} \|u\|^{1/2}_{L^{\frac{2(n+1)}{n-1}}(M)}\\
 +h\|u\|_{L^{\frac{2(n+1)}{n-1}}(M)} + h^{\frac{2}{n+1}}\|f\|_{L^{\frac{2(n+1)}{n+3}}(M)}\bigg)\\
 \le C  \bigg( \frac{h^{\frac{2}{n+1}}}{2\varepsilon}\|f\|_{L^{\frac{2(n+1)}{n+3}}(M)} +\frac{\varepsilon}{2} \|u\|_{L^{\frac{2(n+1)}{n-1}}(M)} \\
+ h\|u\|_{L^{\frac{2(n+1)}{n-1}}(M)} + h^{\frac{2}{n+1}}\|f\|_{L^{\frac{2(n+1)}{n+3}}(M)}. 
 \bigg)
\end{align*}
 Taking $\varepsilon>0$ and $h_0>0$ sufficiently small but fixed, we obtain that for all $h\in (0,h_0]$,
\[
\|u\|_{L^{\frac{2(n+1)}{n-1}}(M)}\le Ch^{\frac{2}{n+1}}\|f\|_{L^{\frac{2(n+1)}{n+3}}(M)}.
\]
This completes the proof of the resolvent estimate \eqref{eq_Lap_rev_0}  in the region where $\Im z=\delta$ and $|\Re z|\ge C$, for some $C>0$ large enough. 

\subsection{Easy spectral regions} In the region where $\Im z=\delta$ and $|\Re z|\le C$, the estimate \eqref{eq_Lap_rev_0}  is a consequence of the uniform estimate \eqref{eq_resolvent_est_laplacian_z}, in view of the embeddings
\begin{equation}
\label{eq_Lap_rev_embedding_new}
L^{\frac{2n}{n-2}}(M)\subset L^{\frac{2(n+1)}{n-1}}(M)\subset L^2(M)\subset L^{\frac{2(n+1)}{n+3}}(M)\subset L^{\frac{2n}{n+2}}(M).
\end{equation}
When $z\in \Xi_\delta$ and $\Im z\ge 2 |\Re z|$, we see that $\Re(z^2)=(\Re z)^2-(\Im z)^2<0$, and therefore,
\begin{equation}
\label{eq_Lap_rev_12}
\|(-\Delta_g-z^2)^{-1}\|_{L^2(M)\to L^2(M)}=\frac{1}{\text{dist}(z^2,\text{Spec}(-\Delta_g))}=\frac{1}{|z|^2}. 
\end{equation}
In this region the estimate \eqref{eq_Lap_rev_0} follows therefore by the Riesz--Thorin interpolation theorem between  \eqref{eq_Lap_rev_12} and \eqref{eq_resolvent_est_laplacian_z}. To establish the estimate \eqref{eq_Lap_rev_0}  in the remaining regions given by $\delta<\Im z<2\Re z$ and $\delta<\Im z<-2\Re z$, we apply the Phragm\'en--Lindel\"of principle  to the holomorphic function
\[
z\mapsto z^{\frac{2}{n+1}}((-\Delta_g-z^2)^{-1}u, v)_{L^2(M)}, \quad u,v\in C^\infty(M),
\] 
in these regions. The proof of Remark \ref{rem_main} is complete.

\section{Damped wave equation. Proof of Theorem \ref{thm_main_2}}

\label{sec_Damped}

Here we shall revisit the proof of Theorem \ref{thm_main}, to prove Theorem \ref{thm_main_2}.  For $u\in C^\infty(M)$, we write
\begin{equation}
\label{eq_dw_1}
Pu=P(\tau)u=(-\Delta_g+2i\tau a(x)-\tau^2)u=f.
\end{equation}
 As in the case of the Laplacian, the proof of the estimate
\eqref{eq_resolvent_est_damped_wave_equation} will consist of several different cases, depending on the location of the spectral parameter $\tau$ in the region $\Pi_{\delta,V}$ of the complex plane, defined by  \eqref{eq_resolvent_est_damped_wave_equation_region}. Let us start with the most significant region. 

\subsection{Spectral Region I}  Assume that $\tau\in \C$ is such that $A_+ +\delta\le \Im \tau:=\beta\le \sup a+\delta$ and $|\Re\tau|$ is large.  Then it will be convenient to make a semiclassical reduction in \eqref{eq_dw_1} so that we let 
\[
|\Re \tau|=\frac{1}{h}, 
\]
where $0< h\ll 1$ is a semiclassical parameter. It follows from \eqref{eq_dw_1} that 
\begin{equation}
\label{eq_dw_2}
h^2P u=(-h^2\Delta_g-1 -h(\pm 2i\beta-h\beta^2)+2i h (\pm 1+ih\beta)a(x))u=h^2 f.
\end{equation}
We shall consider the case $\Re \tau>0$, as the other case can be treated in the same way. 

A crucial step now is the derivation of an a priori estimate, which is similar to the estimate \eqref{eq_2_3_0} in the case of the Laplacian.  Once this estimate has been established, the rest of the proof will follow along the same lines as the proof of Theorem \ref{thm_main}. When proving the a priori estimate, following an argument of \cite[Section 2]{Sjostrand_2000}, we shall conjugate the operator $h^2P$ by an elliptic self-adjoint operator $Q=\text{Op}_h^w(e^q)$, where $q\in S^0(T^*M)$ is to be chosen. Notice that for all $h>0$ small enough, we have $Q^{-1}=\text{Op}_h^w(e^{-q})+hR_0$ where $R_0\in \text{Op}_h^w(S^{-1})$. Letting $p(x,\xi)$ be  the semiclassical principal symbol of $-h^2\Delta_g$, by Proposition \ref{prop_composition_of_operators}, we get
\begin{align*}
Q^{-1}(-h^2\Delta_g) Q&=-h^2\Delta_g +Q^{-1}[-h^2\Delta_g, Q]\\
&=-h^2\Delta_g +Q^{-1}\bigg(\frac{h}{i}\text{Op}_h^w(e^q H_p(q))+h^2R_1\bigg)\\
&=-h^2\Delta_g -ih\text{Op}_h^w(H_p(q))+h^2R_2,
\end{align*}
and 
\[
Q^{-1} a(x) Q=a(x)+hR_3,
\]
with $R_1,R_2\in \text{Op}_h^w(S^0)$ and $R_3\in \text{Op}_h^w(S^{-1})$,  and therefore, 
\begin{equation}
\label{eq_conj_op}
\begin{aligned}
Q^{-1}(h^2P)Q=-h^2\Delta_g -1+ih\text{Op}^w_h(   2 a(x)-H_p(q))-2i h\beta+h^2R_4, 
\end{aligned}
\end{equation}
with $R_4\in \text{Op}_h^w(S^0)$. 
It also follows from \eqref{eq_dw_2} that 
\begin{equation}
\label{eq_damp_3_Q}
Q^{-1}(h^2P)Qv=h^2Q^{-1}f, 
\end{equation}
where 
\begin{equation}
\label{eq_def_v_damp}
v=Q^{-1}u.
\end{equation}

Setting 
\[
\tilde q(x,\xi)=\int_0^T\bigg(\frac{t}{T}-1\bigg) a(\text{exp}(t H_p)(x,\xi))dt+\int_{-T}^0\bigg(1+\frac{t}{T}\bigg) a(\text{exp}(t H_p)(x,\xi))dt,
\]
we check that 
\[
2 a(x)-H_p(\tilde q)= 2 \langle a\rangle_T\quad \text{on}\quad T^*M. 
\]
Let us choose $\varepsilon>0$ sufficiently small but fixed, and let $\varphi\in C^\infty_0(T^*M)$ be such that $\varphi=1$ on $p^{-1}([1-\varepsilon, 1+\varepsilon])$. Setting $q=\varphi\tilde q$, we see that $q\in S^0(T^*M)$ and 
\begin{equation}
\label{eq_damp_3_0_form}
2 a(x)-H_p(q)= 2 \langle a\rangle_T\quad \text{on}\quad p^{-1}((1-\varepsilon, 1+\varepsilon)). 
\end{equation}

It follows from \eqref{eq_conj_op}  and \eqref{eq_damp_3_0_form} that 
\begin{equation}
\label{eq_conj_op_1}
\begin{aligned}
Q^{-1}(h^2P)Q=-h^2\Delta_g -1+ih\text{Op}_h( \hat a_T)-2i h \beta+h^2R_4,
\end{aligned}
\end{equation}
where $\hat a_T\in S^0(T^*M)$ is such that 
\begin{equation}
\label{eq_damp_3}
\hat a_T=2\langle a\rangle_T \quad \text{on}\quad p^{-1}((1-\varepsilon, 1+\varepsilon)).
\end{equation}

Let $1-\varepsilon<E<1+\varepsilon$. Then by the homogeneity property of the $H_p$-flow, we have
\[
\sup_{p^{-1}(E)}\langle a \rangle_T=\sup_{p^{-1}(1)}\langle a\rangle_{\sqrt{E}T},
\]
and therefore, 
\[
\lim_{T\to \infty}\sup_{p^{-1}(E)}\langle a \rangle_T=\lim_{T\to\infty} \sup_{p^{-1}(1)}\langle a\rangle_{\sqrt{E}T}=A_+,
\] 
locally uniformly in $E>0$.  Hence, choosing $T$ sufficiently large but fixed, depending on $\delta>0$, we get 
\begin{equation}
\label{eq_damp_4}
\langle a \rangle_T(x,\xi)\le A_++\frac{\delta}{2}, 
\end{equation}
for all $(x,\xi)\in p^{-1}([1-\varepsilon, 1+\varepsilon])$. 

 It follows from \eqref{eq_conj_op_1} that 
\begin{equation}
\label{eq_damp_1}
\begin{aligned}
\Im(Q^{-1}(h^2P)Q)=h\Re \text{Op}_h^w(\hat a_T)-2h\beta+h^2R_5,
\end{aligned}
\end{equation}
where $R_5\in \text{Op}_h^w(S^0)$. Using Proposition \ref{prop_composition_of_operators} and the fact that $\hat a_T$ is real-valued, we get 
\begin{equation}
\label{eq_damp_2}
\begin{aligned}
\Re \text{Op}_h^w(\hat a_T)=\frac{1}{2} \big(\text{Op}_h^w(\hat a_T)+ \text{Op}_h^w(\hat a_T)^* \big)=\text{Op}_h^w(\hat a_T)+hR_6,
\end{aligned}
\end{equation}
where $R_6\in \text{Op}_h^w(S^{-1})$. We conclude from \eqref{eq_damp_1} and \eqref{eq_damp_2} that 
\begin{equation}
\label{eq_damp_5}
\Im(Q^{-1}(h^2P)Q)=h\text{Op}_h^w(\hat a_T-2\beta)+h^2R_7,
\end{equation}
where $R_7\in \text{Op}_h^w( S^0)$. 

Using the fact that $\beta\ge A_++\delta$, and \eqref{eq_damp_3}, \eqref{eq_damp_4}, we get
\[
2\beta-\hat a_T=2(\beta-\langle a\rangle_T)\ge \delta\quad  \text{on}\quad p^{-1}((1-\varepsilon, 1+\varepsilon)).
\]
Let $0\le \chi\in C^\infty_0(p^{-1}((1-\varepsilon,1+\varepsilon)))$ be such that $\chi=1$ near $p^{-1}(1)$. An application of the semiclassical  microlocalized version of  G{\aa}rding's inequality, see Theorem \ref{thm_Garding_ineq}, gives 
\begin{equation}
\label{eq_damp_6}
\begin{aligned}
\big(  \text{Op}_h^w(2\beta-\hat a_T) \text{Op}_h^w(\chi)v, \text{Op}_h^w(\chi)v\big)_{L^2(M)}\ge &\frac{\delta}{2}\| \text{Op}_h^w(\chi)v\|^2_{L^2(M)}\\
&-\mathcal{O}(h^\infty)\|v\|^2_{L^2(M)},
\end{aligned}
\end{equation}
for all $0<h$ small enough. 

Using \eqref{eq_damp_5} and \eqref{eq_damp_6}, we obtain that
\begin{equation}
\label{eq_damp_8_0}
\begin{aligned}
\frac{\delta h}{4}\| \text{Op}_h^w(\chi)v\|^2_{L^2(M)}&-\mathcal{O}(h^\infty)\|v\|^2_{L^2(M)} \\
&\le -\big(\Im(Q^{-1}(h^2P)Q) \text{Op}_h^w(\chi)v, \text{Op}_h^w(\chi)v\big)_{L^2(M)}\\
&= -\Im \big((Q^{-1}(h^2P)Q) \text{Op}_h^w(\chi)v, \text{Op}_h^w(\chi)v\big)_{L^2(M)}\\
&\le \big|\big(Q^{-1}(h^2P)Q \text{Op}_h^w(\chi)v, \text{Op}_h^w(\chi)v\big)_{L^2(M)}\big|,
\end{aligned}
\end{equation}
for all $0<h$ small enough. 

In view of \eqref{eq_damp_3_Q} we have
\begin{equation}
\label{eq_damp_8_0_1}
Q^{-1}(h^2P)Q \text{Op}_h^w(\chi)v= h^2 \text{Op}_h^w(\chi) Q^{-1}f+ [Q^{-1}(h^2P)Q, \text{Op}_h^w(\chi)]v.
\end{equation}
Hence, by  H\"older's inequality and the uniform boundedness of the operators $\text{Op}_h^w(\chi)$ and  $Q^{-1}$ in $L^p$ spaces with  $1<p<\infty$, we get 
\begin{equation}
\label{eq_damp_8_1}
h^2\big|\big(  \text{Op}_h^w(\chi) Q^{-1}f, \text{Op}_h^w(\chi)v\big)_{L^2(M)}\big|\le \mathcal{O}(h^2)\|f\|_{L^{\frac{2n}{n+2}}(M)}\|v\|_{L^{\frac{2n}{n-2}}(M)}.
\end{equation}

To estimate the scalar product $\big|\big(  [Q^{-1}(h^2P)Q, \text{Op}_h^w(\chi)]v,  \text{Op}_h^w(\chi)v \big)_{L^2(M)}\big|$ involving the commutator, we shall argue as follows. Let $\chi_1\in C^\infty_0(p^{-1}((1-\varepsilon, 1+\varepsilon)))$ be such that $\chi_1=1$ near $p^{-1}(1)$ and such that $\supp(\chi_1)$ is contained in the interior of the set where $\chi=1$. 
By \cite[Appendix A]{Sjostrand_Vodev_1997}, see also \cite[Theorem 9.5]{Zworski_book}, we know that $\text{WF}_h([Q^{-1}(h^2P)Q, \text{Op}_h^w(\chi)])$ is a compact subset of $\supp(\chi)$ such that
\[
\text{WF}_h([Q^{-1}(h^2P)Q, \text{Op}_h^w(\chi)])\cap \{(x,\xi): \chi(x,\xi)=1\}^{\circ}=\emptyset,
\]
where $\{\cdot\}^{\circ}$ denotes the interior of the set. 
Hence, 
\[
\text{WF}_h([Q^{-1}(h^2P)Q, \text{Op}_h^w(\chi)])\cap \text{WF}_h(\text{Op}_h^w(\chi_1))=\emptyset,
\]
and therefore, 
\[
[Q^{-1}(h^2P)Q, \text{Op}_h^w(\chi)]\text{Op}_h^w(\chi_1)=\mathcal{O}(h^\infty): H^{s_1}(M)\to H^{s_2}(M),
\]
for any $s_1, s_2\in \R$. Thus, 
\begin{equation}
\label{eq_damp_7}
\begin{aligned}
\big|\big(&  [Q^{-1}(h^2P)Q, \text{Op}_h^w(\chi)]v,  \text{Op}_h^w(\chi)v \big)_{L^2(M)}\big|\le \mathcal{O}(h^\infty)\|v\|_{L^2(M)}^2\\
&+ \big|\big( [Q^{-1}(h^2P)Q, \text{Op}_h^w(\chi)] (1-\text{Op}_h^w(\chi_1))v,  \text{Op}_h^w(\chi)v \big)_{L^2(M)}\big|.
\end{aligned}
\end{equation}
In view of \eqref{eq_conj_op}, we know that the operator $Q^{-1}(h^2P)Q$ is elliptic on $\supp(1-\chi_1)$, and thus, there exists a parametrix $E\in \text{Op}_h^w(S^{-2}(T^*M))$ such that 
\[
E Q^{-1}(h^2P)Q=1-\text{Op}_h^w(\chi_1)+R,
\]
where 
\[
R\in \cap_{N\ge 0, M\ge 0}h^N\text{Op}_h^w(S^{-M}).
\]
Applying $E$ to  \eqref{eq_damp_3_Q}, we get
\begin{equation}
\label{eq_damp_8_3}
(1-\text{Op}_h^w(\chi_1))v=h^2 EQ^{-1}f-Rv.
\end{equation}
Using \eqref{eq_damp_8_3} together with the fact that 
\[
 [Q^{-1}(h^2P)Q, \text{Op}_h^w(\chi)]\in h\text{Op}_h^w(S^{-\infty}),
\]
we get  
\begin{equation}
\begin{aligned}
\label{eq_damp_8}
\big|\big( [Q^{-1}(h^2P)Q,& \text{Op}_h^w(\chi)] (1-\text{Op}_h^w(\chi_1))v,  \text{Op}_h^w(\chi)v \big)_{L^2(M)}\big|\\
&\le h^2 |\big( [Q^{-1}(h^2P)Q, \text{Op}_h^w(\chi)]   EQ^{-1}f  ,  \text{Op}_h^w(\chi)v \big)_{L^2(M)}\big|\\
&+\big|\big( [Q^{-1}(h^2P)Q, \text{Op}_h^w(\chi)] Rv,  \text{Op}_h^w(\chi)v \big)_{L^2(M)}\big|
\\
&\le \mathcal{O}(h^3)\|f\|_{L^{\frac{2n}{n+2}}(M)}\|v\|_{L^{\frac{2n}{n-2}}(M)}+\mathcal{O}(h^\infty)\|v\|_{L^2(M)}^2.
\end{aligned}
\end{equation}
We conclude from \eqref{eq_damp_7} and \eqref{eq_damp_8} that
\begin{equation}
\label{eq_damp_8_2}
\begin{aligned}
\big|\big(  [Q^{-1}(h^2P)Q, \text{Op}_h^w(\chi)]v,  \text{Op}_h^w(\chi)v \big)_{L^2(M)}\big|&\le \mathcal{O}(h^\infty)\|v\|_{L^2(M)}^2\\
&+\mathcal{O}(h^3)\|f\|_{L^{\frac{2n}{n+2}}(M)}\|v\|_{L^{\frac{2n}{n-2}}(M)}.
\end{aligned}
\end{equation}
It follows from \eqref{eq_damp_8_0}, \eqref{eq_damp_8_0_1}, \eqref{eq_damp_8_1}, and \eqref{eq_damp_8_2} that 
\begin{equation}
\label{eq_damp_9}
\| \text{Op}_h^w(\chi)v\|^2_{L^2(M)}\le \mathcal{O}(h)\|f\|_{L^{\frac{2n}{n+2}}(M)}\|v\|_{L^{\frac{2n}{n-2}}(M)}  +\mathcal{O}(h^\infty)\|v\|^2_{L^2(M)}. 
\end{equation}
Using the fact that the operator $Q^{-1}(h^2P)Q$ is elliptic on $\supp(1-\chi)$, we get 
\eqref{eq_damp_8_3} with $\chi$ in place of $\chi_1$. This implies that 
\begin{equation}
\label{eq_damp_10}
\begin{aligned}
\|(1-\text{Op}_h^w(\chi))v\|^2_{L^2(M)}&\le h^2 |(EQf, (1-\text{Op}_h^w(\chi))v)_{L^2(M)}|\\
&+|(Rv, (1-\text{Op}_h^w(\chi))v)_{L^2(M)}|\\
& \le \mathcal{O}(h^2)\|f\|_{L^{\frac{2n}{n+2}}(M)}\|v\|_{L^{\frac{2n}{n-2}}(M)} 
+\mathcal{O}(h^\infty)\|v\|^2_{L^2(M)}.
\end{aligned}
\end{equation}
The estimates \eqref{eq_damp_9} and \eqref{eq_damp_10} yield the following a priori estimate,
\begin{equation}
\label{eq_damp_11}
\|v\|^2_{L^2(M)}\le \mathcal{O}(h)\|f\|_{L^{\frac{2n}{n+2}}(M)}\|v\|_{L^{\frac{2n}{n-2}}(M)}, 
\end{equation}
for all $h>0$ small enough, which is similar to \eqref{eq_2_3_0} in the case of the Laplacian. 

Now as a consequence of 
\eqref{eq_damp_3_Q} and \eqref{eq_conj_op_1}, we obtain the following equality, which is similar to \eqref{eq_2_2},
\[
(-h^2\Delta_g-1)v=h^2Q^{-1}f+h \text{Op}_h^w(r_7)v, 
\]
where $r_7=2i (A_+ +\delta)-i\hat a_T-hr_3\in S^0(T^*M)$. Relying on the microlocal factorization of the semiclassical principal symbol of  the operator $-h^2\Delta_g-1$, the semiclassical Strichartz estimates, and a parametrix in the elliptic region, similarly to the discussion of the crucial spectral region in the proof of  Theorem \ref{thm_main}, we obtain the estimate
\[
\|v\|_{L^{\frac{2n}{n-2}}(M)}\le C(h^{-\frac{1}{2}}\|v\|_{L^2(M)}+h\|v\|_{L^{\frac{2n}{n-2}}(M)} +\|f\|_{L^{\frac{2n}{n+2}}(M)}),
\]
which is the same as the estimate \eqref{eq_2_30_2damp} in the proof of Theorem \ref{thm_main}. Using the a priori estimate \eqref{eq_damp_11}, we get that for all $h>0$ small enough, 
\[
\|v\|_{L^{\frac{2n}{n-2}}(M)}\le C \|f\|_{L^{\frac{2n}{n+2}}(M)}.
\]
This together with \eqref{eq_def_v_damp} completes the proof of the uniform resolvent estimate \eqref{eq_resolvent_est_damped_wave_equation}  in the spectral region where  $A_++\delta\le \Im \tau\le \sup a+\delta$ and $|\Re \tau|\ge L_1$, for some $L_1>0$ large enough. 

\subsection{Spectral region II}  Let $\tau\in \C$ be such that  $\inf a-\delta\le \Im \tau\le A_--\delta$ and $|\Re\tau|$ is large. This region can be treated in the same way as the first region. 

\subsection{Spectral region III} Let $\tau\in \C$ be such that $|\Re\tau|\le \frac{1}{2}|\Im\tau|$ and $|\Im \tau|$ sufficiently large, depending on the damping coefficient $a$. Multiplying \eqref{eq_dw_1} by $\overline{u}$, integrating by parts and taking the real part, we get
\begin{equation}
\label{eq_spec_III_new}
\begin{aligned}
\|\nabla_g u\|_{L^2(M)}^2+((\Im\tau)^2-(\Re\tau)^2)&\|u\|_{L^2(M)}^2\\
&=2\Im \tau \int a|u|^2dV_g+\Re(f,u)_{L^2(M)},
\end{aligned}
\end{equation}
where $dV_g$ is the Riemannian volume element of $M$. Using that $|\Re\tau|\le \frac{1}{2}|\Im\tau|$, we obtain that 
\begin{align*}
\|\nabla_g u\|_{L^2(M)}^2+\frac{3}{4}(\Im\tau)^2&\|u\|_{L^2(M)}^2\\
&\le 2|\Im \tau| \|a\|_{L^\infty(M)}\|u\|^2_{L^2(M)}+\|f\|_{H^{-1}(M)}\|u\|_{H^1(M)}.
\end{align*}
Assuming that $|\Im \tau|\ge L_2$, where $L_2>0$ is sufficiently large constant, depending on $\|a\|_{L^\infty(M)}$, so that 
\[
|\Im \tau|\bigg(\frac{3}{4}|\Im \tau|-2\|a\|_{L^\infty}\bigg)\ge 1,
\] 
we get $\|u\|_{H^1(M)}\le \|f\|_{H^{-1}(M)}$, and therefore, by Sobolev's embedding, we obtain  the uniform resolvent estimate \eqref{eq_resolvent_est_damped_wave_equation} in this region.  

\subsection{Spectral region IV} Let $L=\max\{L_1,L_2\}$, and  let $V$ be an open neighborhood of the set $\text{Spec}(P(\tau))\cap \{\tau\in \C: |\Re \tau|\le L, |\Im \tau|\le 2L\}$. Then the set $K=\{\tau\in \C: |\Re \tau|\le L, |\Im \tau|\le 2L\}\setminus V$ is compact. Once we know that the following uniform estimate holds,
\begin{equation}
\label{eq_damp_12}
\| (-\Delta_g+2i a(x)\tau -\tau^2)^{-1}\|_{H^{-1}(M)\to H^1(M)}\le C,
\end{equation}
for all $\tau\in K$ with a constant $C>0$, independent of $\tau$, the uniform resolvent estimate \eqref{eq_resolvent_est_damped_wave_equation}  in the case $\tau\in K$ is a consequence of \eqref{eq_damp_12} and  Sobolev's embedding.  To show \eqref{eq_damp_12}, first observe that the operator 
\[
P(\tau)=-\Delta_g+2i a(x)\tau -\tau^2: H^1(M)\to H^{-1}(M)
\]
is Fredholm of index zero, and it follows from the analytic Fredholm theory that the inverse $P(\tau)^{-1}$ exists for $\tau\in \C\setminus\text{Spec}(P(\tau))$ and moreover, 
\[
\C\setminus\text{Spec}(P(\tau))\ni \tau\mapsto P(\tau)^{-1}\in \mathcal{L}(H^{-1}(M),H^{1}(M))
\]
is holomorphic. Hence,  the function
\[
\C\setminus\text{Spec}(P(\tau))\ni \tau \mapsto \|P^{-1}(\tau)\|_{\mathcal{L}(H^{-1}(M), H^{1}(M))}
\]
is continuous, and thus, bounded on the compact set $K$.  The uniform estimate  \eqref{eq_damp_12} follows.

\subsection{Spectral region V} Let us finally discuss the remaining four portions of the spectral $\tau$--plane, 
\begin{align*}
\Sigma_1&=\{\tau\in \C:  \sup a+\delta< \Im \tau< 2 \Re \tau, \Re\tau\ge  L\}, \\
\Sigma_2&=\{ \tau\in \C:   \sup a+\delta<\Im \tau<-2\Re \tau, \Re\tau\le - L\},\\
\Sigma_3&=\{\tau \in \C:  -2\Re \tau<\Im\tau<\inf a-\delta, \Re \tau\ge L\},\\
\Sigma_4&=\{\tau\in \C:  2\Re\tau<\Im \tau<\inf a-\delta, \Re\tau\le -L\}. 
\end{align*}
Here the following $L^2$ resolvent estimates for the stationary damped wave operator, obtained by integration by parts, will be important,
\begin{align}
\label{eq_resolvent_damp_L_2}
\| (-\Delta_g+2i a(x)\tau-\tau^2)^{-1}f\|_{L^2(M)}\le& \frac{1}{2|\Re  \tau|(\Im \tau -\sup a)}\|f\|_{L^2(M)}, \nonumber \\
&\Re\tau\ne 0,\quad \Im \tau>\sup a, \\
\| (-\Delta_g+2i a(x)\tau-\tau^2)^{-1}f\|_{L^2(M)}\le& \frac{1}{2|\Re  \tau|(\inf a- \Im \tau)}\|f\|_{L^2(M)}, \nonumber \\
&\Re\tau\ne 0,\quad \Im \tau<\inf a \nonumber.
\end{align}

Without loss of generality we shall consider the region $\Sigma_1$. Let $u,v\in C^\infty(M)$ be fixed and let 
\[
F(\tau)=( (-\Delta_g+2i a(x)\tau-\tau^2)^{-1} u,v  )_{L^2(M)}, \quad \tau\in \Sigma_1. 
\]
Then the function $F(\tau)$ is analytic in $\Sigma_1$ and continuous in $\overline{\Sigma}_1$. Furthermore, for any 
$\tau\in \Sigma_1$, using \eqref{eq_resolvent_damp_L_2}, we get
\[
|F(\tau)|\le \frac{1}{2L\delta}\|u\|_{L^2(M)}\|v\|_{L^2(M)}.
\]
We have also shown that for any $\tau\in \p \Sigma_1$, the following estimate holds,
\begin{equation}
\label{eq_damp_14}
|F(\tau)|\le C\|u\|_{L^\frac{2n}{n+2}(M)}\|v\|_{L^\frac{2n}{n+2}(M)},
\end{equation}
with $C$ being independent of $\tau$, and thus, by the Phragm\'en-Lindel\"of  principle, we have the estimate \eqref{eq_damp_14} for all $\tau\in \Sigma_1$.  Hence, for any $\tau\in \Sigma_1$,  we have
\begin{align*}
\|(-\Delta_g+&2i a(x)\tau-\tau^2)^{-1} u\|_{L^\frac{2n}{n-2}(M)}\\
&=\sup_{v\in C^\infty(M),v\ne 0}\frac{| ( (-\Delta_g+2i a(x)\tau-\tau^2)^{-1} u,v  )_{L^2(M)}|}{\|v\|_{L^{\frac{2n}{n+2}}(M)}}
\le C\|u\|_{L^\frac{2n}{n+2}(M)}, 
\end{align*}
which completes the proof of  the uniform resolvent estimate \eqref{eq_resolvent_est_damped_wave_equation} for $\tau\in \Sigma_1$. The proof of Theorem \ref{thm_main_2} is complete.

\section{Damped wave equation. Proof of Remark \ref{rem_main_2} }
\label{sec_damped_Remark}

Let us first observe that to establish  the estimate \eqref{eq_rem_main_2} it suffices to prove the following bound, 
\begin{equation}
\label{eq_damp_rev_0}
\|u\|_{L^{\frac{2(n+1)}{n-1}}(M)}\le C |\tau |^{-\frac{2}{n+1}}\|P(\tau) u\|_{L^{\frac{2(n+1)}{n+3}}(M)},
\end{equation}
valid for $u\in C^\infty(M)$ and $\tau\in\Pi_{\delta,V}$. When discussing the derivation of \eqref{eq_damp_rev_0}, we shall use the same notation as in Section \ref{sec_Damped}.

\subsection{Spectral Region I}  Here we assume that $\tau\in \C$ is such that $A_+ +\delta\le \Im \tau:=\beta\le \sup a+\delta$ and $|\Re\tau|$ is large.  The discussion in Sections \ref{sec_Damped} and \ref{sec_Laplace_Remark} shows that to establish \eqref{eq_damp_rev_0} in this region, it suffices to obtain the following a priori estimate 
\begin{equation}
\label{eq_damp_rev_1}
\|u\|^2_{L^2(M)}\le \mathcal{O}(h) \|f\|_{L^{\frac{2(n+1)}{n+3}}(M)}\|u\|_{L^{\frac{2(n+1)}{n-1}}(M)},
\end{equation}
cf. \eqref{eq_damp_11}, which follows by a straightforward inspection of the conjugation argument in Section \ref{sec_Damped}. 

\subsection{Spectral Region II}  Let $\tau\in \C$ be such that  $\inf a-\delta\le \Im \tau\le A_--\delta$ and $|\Re\tau|$ is large. This region can be treated in the same way as the first region. 

\subsection{Spectral region III} Let $\tau\in \C$ be such that $|\Re\tau|\le \frac{1}{2}|\Im\tau|$ and $|\Im \tau|$ sufficiently large, depending on the damping coefficient $a$. Here the estimate \eqref{eq_damp_rev_0} follows by the Riesz--Thorin interpolation theorem between the uniform estimate  \eqref{eq_resolvent_est_damped_wave_equation}  and the following $L^2$ bound for the resolvent,
\begin{equation}
\label{eq_damp_rev_2}
\|P(\tau)^{-1}\|_{L^2(M)\to L^2(M)}=\mathcal{O}\bigg(\frac{1}{|\tau|^2}\bigg). 
\end{equation}
When  checking \eqref{eq_damp_rev_2}, we observe that  \eqref{eq_spec_III_new} implies that 
 \[
 \frac{3}{4}(\Im \tau)^2\|u\|^2_{L^2(M)}\le 2|\Im \tau|\|a\|_{L^\infty}\|u\|^2_{L^2(M)}+\|f\|_{L^2(M)}\|u\|_{L^2(M)}.
 \]
Assuming that $|\Im \tau|\ge 4\|a\|_{L^\infty}:=L_2$, we get 
\[
\frac{1}{4}|\Im \tau |^2\|u\|_{L^2(M)}\le \|f\|_{L^2(M)},
\]
showing \eqref{eq_damp_rev_2}. 

\subsection{Spectral region IV}  The estimate \eqref{eq_damp_rev_0} in the compact spectral region $K=\{\tau\in \C: |\Re \tau|\le L, |\Im \tau|\le 2L\}\setminus V$  follows from the uniform estimate  \eqref{eq_resolvent_est_damped_wave_equation}  and the embedding \eqref{eq_Lap_rev_embedding_new}. 

\subsection{Spectral region V} Here the estimate \eqref{eq_damp_rev_0} is obtained by an application of the Phragm\'en--Lindel\"of principle  to the holomorphic function
\[
\tau\mapsto \tau^{\frac{2}{n+1}}(P(\tau)^{-1}u, v)_{L^2(M)}, \quad u,v\in C^\infty(M)
\] 
in this region. 
 The proof of Remark \ref{rem_main_2} is complete. 

\section{Note added in proof}
After submitting a revised version of this paper, we learned that in the paper "Endpoint resolvent estimates for compact Riemannian manifolds" by R. Frank and L. Schimmer, Journal of Functional Analysis, to appear, an alternative proof of the estimate \eqref{eq_rem_main}, in the endpoint case $p =2(n+1)/(n+3)$, is given, by a completely different method.

\section*{Acknowledgements}
N.B. and K.K. would like to thank Setsuro Fujii\'e for the very kind invitation and hospitality at the Ritsumeikan University, Kyoto, in July 2014, where this work was initiated.  D. DSF. and K.K. would also like to thank the Institut Henri Poincar\'e, Paris, where the work was completed. 
We are very grateful to the referees for helpful suggestions and remarks. We are particularly indebted to one of the referees for the 
very interesting suggestion to generalize our estimates, by establishing also the bounds in Remark \ref{rem_main_2} and Remark \ref{rem_main}, 
which has improved the paper. We are furthermore most grateful to this referee for pointing out the work \cite{Koch_Tataru_1995} to us, which has led to Remark \ref{rem_main_2_1}. The research of N.B. is partially supported by  the Agence Nationale de la Recherche through ANR-13-BS01-0010-03 (ANA\'E)  and  201-BS01019 01 (NOSEVOL). The research of D.DSF is partially supported by the Agence Nationale de la Recherche through ANR-13-JS01-0006 (\textit{iproblems}). The research of K.K. is partially supported by the National Science Foundation (DMS 1500703).

\end{document}